\documentclass[a4j,12pt,leqno]{amsart}
\usepackage{amsmath,amssymb}
\usepackage[all,cmtip]{xy}
\usepackage[dvipdfmx]{graphicx}
\usepackage{mathrsfs}
\makeatletter
\renewcommand{\section}{\@startsection
{section}{1}{0mm}{5mm}{2mm}{\raggedright\bfseries}}
  \@addtoreset{equation}{section}
\makeatother

\theoremstyle{plain} 
\newtheorem{theorem}{\indent\sc Theorem}[section] 
\newtheorem{lemma}[theorem]{\indent\sc Lemma}
\newtheorem{corollary}[theorem]{\indent\sc Corollary}
\newtheorem{proposition}[theorem]{\indent\sc Proposition}

\theoremstyle{definition} 
\newtheorem{definition}[theorem]{\indent\sc Definition}
\newtheorem{remark}[theorem]{\indent\sc Remark}

\def\bkappa{{\boldsymbol \kappa}}
\def\cyc{{\mathrm{cyc}}}
\def\pol{{\mathrm{pol}}}
\def\muprime{{\mu(\Z_p^{ur})}}

\def\padic{{p\mathchar`-\mathrm{adic}}}
\def\ur{\mathrm{ur}}
\def\Norm{{\mathcal N}}

\def\proj{{\mathbf{p}}}
\def\bTrace{{\overline{\mathcal S}}}
\def\cLog{{\mathcal L}}
\def\nyoroto{{\rightsquigarrow}}
\def\tilchi{{\tilde{\chi}}}

\def\Li{{\mathrm{Li}}}

\def\ab{{\mathrm{ab}}}

\def\ff{{\frak f}}

\def\la{{\langle}}
\def\ra{{\rangle}}

\def\ovec#1{\overrightarrow{#1}}

\def\isom{\,{\overset \sim \to  }\,}

\def\Z{{\Bbb Z}}

\def\C{{\Bbb C}}

\def\cU{{\mathcal U}}
\def\cL{{\mathcal L}}

\def\cO{{\mathcal O}}
\def\Q{{\Bbb Q}}

\def\Ker{\mathrm{Ker}}

\def\Gal{\mathrm{Gal}}
\def\Hom{\mathrm{Hom}}

\def\ontomap{\twoheadrightarrow}
\def\Q{{\mathop{\mathbb{Q}}\nolimits}}
\def\P{\mathop{\mathbb{P}}\nolimits}
\def\Z{{\mathop{\mathbb{Z}}\nolimits}}
\def\C{\mathop{\mathbb{C}}\nolimits}
\def\Qac{\overline{\mathbb{Q}}}
\def\Ker{\mathop{\mathrm{Ker}}\nolimits}

\def\barX{\mathop{\mathbb P^1_{\overline{\mathbb Q}}\setminus \{0,1,\infty\}}\nolimits}
\def\barGm{\mathop{\mathbb P^1_{\overline{\mathbb Q}}\setminus \{0,\infty\}}\nolimits}
\def\({\mathop{[\![}\nolimits}
\def\){\mathop{]\!]}\nolimits}

\def\B+{\mathop{B^+_{\mathrm{dR}}}\nolimits}

\def\D+{\mathop{D^0_{\mathrm{dR}}}\nolimits}
\def\<{\mathop{\langle}\nolimits}
\def\>{\mathop{\rangle}\nolimits}

\def\RepCrys_F{\mathop{\mathrm{Rep}_{\Q_p}^{\mathrm{crys}}(G_F)}\nolimits}
\def\Rep_F{\mathop{\mathrm{Rep}_{\Q_p}(G_F)}\nolimits}
\def\IRep_F{\mathop{\mathrm{Ind}\mathrm{Rep}_{\Q_p}^{\mathrm{crys}}(G_F)}\nolimits}

\def\Pet#1{\mathop{\mathscr{P}^{\mathrm{et}}_{#1}}\nolimits}

\def\rec{\mathop{\mathrm{rec}}\nolimits}
\makeatother
\begin{document}

\title[Polylogarithmic analogue of the Coleman-Ihara formula, I]
{Polylogarithmic analogue of \\
the Coleman-Ihara formula, I}
\author{Hiroaki Nakamura, Kenji Sakugawa and Zdzis{\l}aw Wojtkowiak}


\address{Hiroaki Nakamura: 
Department of Mathematics, 
    Graduate School of Science, 
    Osaka University, 
    Toyonaka, Osaka 560-0043, Japan}
\email{nakamura@math.sci.osaka-u.ac.jp}

\address{Kenji Sakugawa: 
Department of Mathematics, 
    Graduate School of Science, 
    Osaka University, 
    Toyonaka, Osaka 560-0043, Japan}
\email{sakugawa.kenji@gmail.com}

\address{Zdzis{\l}aw Wojtkowiak: 
Universit\'e de Nice-Sophia Antipolis, D\'ept. of Math., 
Laboratoire Jean Alexandre Dieudonn\'e, 
U.R.A. au C.N.R.S., No 168, Parc Valrose -B.P.N${}^\circ$ 71,
06108 Nice Cedex 2, France}
\email{wojtkow@math.unice.fr}

\begin{abstract}
The Coleman-Ihara formula expresses Soule's $p$-adic characters 
restricted to $p$-local Galois group as
the Coates-Wiles homomorphism multiplied by $p$-adic $L$-values at positive integers.
In this paper, we show an analogous formula that $\ell$-adic
polylogarithmic characters for $\ell=p$ restrict to the Coates-Wiles 
homomorphism multiplied by Coleman's $p$-adic polylogarithms
at any roots of unity of order prime to $p$. 
\end{abstract}

\maketitle
\tableofcontents
\renewcommand{\thefootnote}{}
\footnote{Revised Version, November 16, 2015}
%\LaTeX -compiled on: \today}

\section{Introduction}
Let $p$ be an odd prime. In his Annals article \cite{I86}, 
Yasutaka Ihara introduced the universal power series 
for Jacobi sums $F_\sigma\in\Z_p[\![u,v]\!]^\times$ 
($\sigma\in G_\Q:=\mathrm{Gal}(\bar\Q/\Q)$) 
and showed a beautiful (local) formula (\cite{I86} Theorem C)
relating its coefficient characters
with $p$-adic $L$-values $L_p(m,\omega^{1-m})$
$(m\ge 3:$ odd) multiplied by the $m$-th Coates-Wiles homomorphisms
when $\sigma$ lies in the $p$-local subgroup of 
$G_{\Q(\mu_{p^\infty})}$. In the last part of \cite{I86} (p.105, (Col2)) 
documented is that Robert Coleman proved that these coefficient characters
are nothing but the restrictions of Soule's cyclotomic elements 
in $H^1(G_{\Q},\Z_p(m))$. This motivated later works by Anderson
\cite{A}, Coleman \cite{C5}, 
Ihara-Kaneko-Yukinari \cite{IKY} toward the explicit (global) 
formula of $F_\sigma$ for all $\sigma\in G_\Q$ (see \cite{I90}).
%Its main body 
Their formula presents particularly a remarkable symmetric form 
on the main part $G_{\Q(\mu_{p^\infty})}$ of $G_\Q$ as:
\begin{equation}
\label{IKY}
F_\sigma(u,v)=
\exp
\left(
\sum_{\substack{m\ge 3 \\ odd}}
\frac{\chi_m(\sigma)}{p^{m-1}-1}
\sum_{\substack{i+j=m \\ i,j\ge 1}} 
\frac{U^iV^j}{i!j!}
\right)
\qquad
(\sigma\in G_{\Q(\mu_{p^\infty})}),
\end{equation}
where $1+u=e^U$, $1+v=e^V$ and 
$\chi_m:G_{\Q(\mu_{p^\infty})}\to \Z_p(m)$ is the $m$-th
Soule character (\cite{Sou}) defined by the properties:
\begin{equation}
\label{Soule}
\bigl(
\prod_{\substack{1\le a<p^n \\ p\nmid a}}
(1-\zeta_{p^n}^a)^{a^{m-1}}
\bigr) ^{\frac{1}{p^n}(\sigma-1)}
=
\zeta_{p^n}^{\chi_m(\sigma)}
\quad (n\ge 1).
\end{equation}
This, together with the above mentioned Ihara's local formula 
\cite{I86} Theorem C, implies Coleman's formula presented
in \cite{I86} p.105 
in the form:
\begin{eqnarray}
\label{FCW}
\frac{\chi_m(\rec(\epsilon))}{(p^{m-1}-1)}
=L_p(m,\omega^{1-m})\phi^{CW}_{m}(\epsilon)
\qquad (\epsilon\in \mathcal{U}_\infty)
%\phi^{CW}_{m,\Q_p}(\epsilon).
\label{eq0.1}
\end{eqnarray}
for $m\geq 3:$ odd. 
Let us quickly explain the notation used here: 
%which are slightly different from notation in \cite{I86}. 
For each $n\ge 1$, we denote by 
${\mathcal U}_n$ the group of principal units of 
$\Q_p(\mu_{p^n})$ and by ${\mathcal U}_\infty=\varprojlim_{n}\mathcal U_n$ 
their norm limit. Let $\Omega _p$ be the maximal abelian 
pro-$p$ extension of $\Q(\mu_{p^\infty})$ unramified outside $p$.
Then, Ihara's power series $\sigma\mapsto F_\sigma(u,v)$ 
factors through $\mathrm{Gal}(\Omega_p/\Q)$.
Now, fix an embedding 
$\Qac\hookrightarrow \overline{\Q_p}$ and 
a coherent system of $p$-power roots of unity $\{\zeta_{p^n}\}_{n\ge 1}$
to identify $\Z_p$ with $\Z_p(m)$. 
This embedding and the local class field theory induce the canonical 
homomorphism $\rec:{\mathcal U}_\infty\hookrightarrow 
\mathrm{Gal}(\Omega_p/\Q(\mu_{p^\infty}))$ called the reciprocity map.
On the other side, the system $\{\zeta_{p^n}\}_n$ determines, for $m\geq 1$, 
the Coates-Wiles homomorphism
$\phi^{CW}_{m}:\mathcal{U}_{\infty}\rightarrow \Z_p$.
The coefficient $L_p(m,\omega^{1-m})$ is the Kubota-Leopoldt 
$p$-adic $L$-value at $m$ with respect to the power of 
the Teichm\"uller character $\omega$.

Indeed, Coleman's paper \cite{C5} proves (\ref{eq0.1}) by applying his theory on 
Hilbert norm residue symbols [C1,2,4] to Jacobi sums which are 
special values $F_\sigma(\zeta_{p^n}^a-1,\zeta_{p^n}^b-1)$ 
at Frobenius elements $\sigma$ over various primes 
in $\Q(\mu_{p^n})$ not dividing $p$. Especially, it relies on the
Tchebotarev density. Consequently, the formula (\ref{FCW}) results
from combination of \cite{C5} and (\ref{IKY}), %\cite{I90}, 
relying on global arithmetic nature of $F_\sigma(u,v)$. 
The global proof certainly enables us to highlight (\ref{FCW}) 
in contexts enriched with many important materials 
of Iwasawa theory (see, e.g., \cite{I-S} \S 3-1). 
However, (\ref{FCW}) itself is essentially of local nature,
concerning the ratio between
the Coates-Wiles homomorphism 
and Soule's character restricted on the local Galois group; 
Passing through the Jacobi sum interpolation 
properties of $F_\sigma(u,v)$ to derive (1.3)
should look rather roundabout.

The purpose of this paper is to give an alternative direct proof of
(\ref{FCW}) and its polylogarithmic variants, where the 
Soule's characters $\chi_m$ in LHS are generalized to the ($l$-adic) 
Galois polylogarithms $\ell i_m(z,\gamma)$ 
(for the case $l=p$)
%:G_{\Q_p(z)}\to \Z_p$ 
introduced in [W1-2].
They are defined as certain coefficients of Galois transforms of a 
defining path $\gamma$ from $\overrightarrow{01}$ to $z$ 
on $\P^1\setminus\{0,1,\infty\}$ (cf.\,\S \ref{section2.1}\,(\ref{eq2.4})).
The values $L_p(m,\omega^{1-m})$ in RHS have obvious generalization 
to $\mathrm{Li}^\padic_m(z)$, the $p$-adic polylogarithms 
of Coleman (see \cite{C2}). 
As a corollary of our main formula (Theorem \ref{fullformula}), 
we obtain:

\begin{theorem}[Corollary \ref{maincor}, Proposition \ref{m=1},
Remark \ref{extensiontoG_F}]
\label{mainthm}
Let $p$ be an odd prime, and $F$ a finite unramified extension of 
$\Q_p$ with %ring of integers $\cO_F$ and 
the Frobenius substitution $\sigma_F\in\Gal(F/\Q_p)$.
Let $F_\infty:=F(\mu_{p^\infty})$ and denote by 
$\phi^{CW}_{m,F}:G_{F_\infty}\to F\otimes\Z_p(m)$ 
the $m$-th Coates-Wiles homomorphism for the local field $F$
(cf.\,Definition \ref{DefCW}).
Then, for any root of unity $z$ contained in $F$,
there is a standard specific path $\ovec{01}\nyoroto z$
such that 
\begin{eqnarray}
\label{maineqn} \label{eq4.2.9}
\ell i_m(z,\gamma)(\sigma)=
\frac{-1}{(m-1)!}
\mathrm{Tr}_{F/\Q_p}
\left(
\left\{
\left(
1-\frac{\sigma_F}{p^m}
\right)
\mathrm{Li}^\padic_m(z)
\right\}
\phi^{CW}_{m,F}(\sigma)
\right)
\end{eqnarray}
holds for 
$m\ge 1$ and 
for $\sigma\in G_{F_\infty}$. 
%$ in the image of the reciprocity map
%$\cU_\infty(F){\longrightarrow} 
%\in G_{F_\infty}^\ab$. 
%associated with the field $F_\infty=F(\zeta_{p^\infty})$.
%the inertia subgroup of the
%transfer image of $G_{\Q_p(\mu_{p^\infty})}^\ab\to %G_{F_\infty}^\ab$.
\end{theorem}

Our approach taken in this Part I is to 
apply the theory of Coleman power series 
(in a direct but different way from [C5]) to 
the so called Koblitz measure that produces 
$\Li_m^\padic(z)$ ([Ko]) on one hand, 
and on another hand, to the explicit formula 
of $\ell i_m(z)$ (\cite{N-W1}) generalizing 
the above Soule's characters (\ref{Soule}).
After the Introduction, in \S 2, we shall recall 
basic setup for Galois polylogarithms
and $p$-adic polylogarithms, and in \S 3, 
we introduce and study a special family of Coleman power series that 
bridges these two kinds of polylogarithms through
Coleman's reciprocity law.
In \S 4, we present a general formula for $p$-adic
polylogarithmic characters on the image 
of $\cU_\infty(F)\overset{\mathrm rec}{\longrightarrow} G_{F_\infty}^\ab$ 
(Theorem \ref{fullformula})
and prove
Theorem \ref{mainthm}.
A (direct) proof of the original Coleman-Ihara formula
(\ref{FCW}) is also obtained as a special case of 
$z=1$, $F=\Q_p$ (Remark \ref{finalremark}).

In a subsequent Part II \cite{NSW}, we study a generalization of 
the above formula to the case of more general 
$z\in (\mathbb P^1\setminus\{0,1,\infty\})({\Z_p^{ur}})$.
%Note that the restriction of $z$ in this region
%$(\mathbb P^1\setminus\{0,1,\infty\})({\Z_p^{\mathrm{ur}}})$
%insures that 
%$\mathrm{Li}^\padic_m(z)$ is independent of a choice of
%Iwasawa $p$-adic logarithm (cf.\,\cite{KN} p.425). 
For the generalization, we switch over to a view on the
polylogarithmic torsors of paths in Deligne's Tannakian
approach [De], and apply a non-commutative lift of 
Bloch-Kato's explicit reciprocity law that reverses 
logarithmic mapping of torsors studied by M.Kim \cite{Kim}, 
M.Olsson \cite{Ol}. Here, we partially rely on \cite{Sak} 
for some technical details. 

We note that Kurihara \cite{Ku} and Gros \cite[II]{Gr} \S IV
gave an alternative local approach to the essentially same formula
as our above Theorem \ref{mainthm} by using syntomic 
cohomology. See \cite{Ku} (2.11),\,(2.12) for the case of $m<p-2$
and a comment in (2.15) on
extension to general case of $m>1$
(in (2.12) of loc.\,cit.,\,``$\mathrm{Tr}_{F/\Q_p}$'' seems lacking in print). 
Compared to the Gros-Kurihara formulation,
our approach in the present paper is of more elementary 
nature and may be useful to find a source reason behind 
the formula (\ref{maineqn}) in certain explicit Coleman 
power series $f_{z,c}(T)$ given in \S 3 below.
Combined with illustrations in loc.\,cit.\,and \cite{KN} 
(in particular, p.425), our above result suggests 
that the $l$-adic Galois polylogarithm 
$\ell i_m(z,\gamma)$ stands nearby a shadow of 
Beilinson's cyclotomic element in $K_{2m-1}(F)\otimes \Q$
at least when $z$ is a root of unity
of order prime to $p$.

\medskip
{\bf Notation: }
In this paper, we let $p$ be a fixed odd prime. 
We fix embeddings $\Qac\hookrightarrow \mathbb{C}$ and 
$\Qac\hookrightarrow \overline{\Q_p}$, which determines
a coherent system of roots of unity 
$\zeta_n\in\mu_n\subset\Qac$ with $\zeta_n=\exp(2\pi i/n)$.
Write $\Z_p^{ur}$ for the ring of integers of the maximal
unramified extension $\Q_p^{ur}$ of $\Q_p$.
The group of roots of unity $\muprime$ is isomorphic to
$\bigcup_{p\nmid N}\mu_N$.

For any unramified extension $F$ of $\Q_p$, we denote the arithmetic 
Frobenius automorphism on $F$ by $\sigma_F:F\to F$. 
We shall use the notation $F_\infty$ to denote 
$F(\mu_{p^\infty})$, while we prefer in this paper
to keep ``$F_n$'' unused to avoid confusion 
(with its custom usage 
`$F_n=F(\mu_{p^{n+1}})$' in Iwasawa theory.) 
We define $\mathcal U_\infty(F)$ 
to be the the norm limit of the 
group of principal units of $F(\mu_{p^n})$. 
The Galois group $\Gal(F_\infty/F)$ 
%(resp. $\Gal(F(\mu_{p^n})/F)$) 
will be written as $G_\infty$.
% (resp. $G_n$).

For any local field $F/\Q_p$, we write $\cO_F$, $k_F$ for the
ring of integers and its residue field respectively.
For any field $F$, we denote by $G_F$ the absolute 
Galois group of $F$, and by $\chi_\cyc$ the 
$p$-adic cyclotomic character 
$G_F\to \Z_p^\times$.

The Bernoulli polynomials and Bernoulli numbers are given by 
$\sum_{n=1}^\infty B_n(X)t^n/n!=te^{Xt}/(e^t-1)$ and 
$B_n:=B_n(0)$. 

%\newpage
%%%%%%%%%%%%%%%%%%%%%%%%%%%%%%%%%%%%%%%%%%%%%%%%%%%%%%%%%%%%%%%%%5
\section{Review of Galois and $p$-adic polylogarithms}

\subsection{Galois polylogarithms}\label{section2.1}
We review $\ell$-adic Galois polylogarithms 
([W1-2]) in the case $\ell =p$. 
Let $\overrightarrow{01}$ be the unit tangential base point
on $\mathbb{P}^1-\{0,1,\infty\}$. We denote 
by $\pi_1$ the maximal pro-$p$ quotient of the 
etale fundamental group 
$\pi_1^{\mathrm{et}}(\barX,\overrightarrow{01})$
and identify $\Z_p(1)$ with that of 
$\pi_1^{\mathrm{et}}(\barGm,\overrightarrow{01})$.
We will always regard these fundamental groups 
as equipped with the actions of $G_\Q$ 
determined by $\ovec{01}$. 

Let $\proj:\pi_1\rightarrow \Z_p(1)$ be the projection homomorphism.
We shall focus on the structure of the quotient group
$\pi_1^\pol:=\pi_1/[\Ker\proj,\Ker\proj]$ 
(the pro-$p$ polylogarithmic quotient)
that has the induced projection 
$\proj':\pi_1^\pol\rightarrow \Z_p(1)$.
By construction, $\Ker\proj'$ is abelian, hence has a conjugate 
action of $\mathrm{Im}(p)=\Z_p(1)$.
In fact, as discussed in \cite[Section 16.11-14]{De}, 
it forms a free $\Z_p\( \Z_p(1)\)$-module generated by 
a generator $y$ of the inertia subgroup over the puncture $1$:
There arises a Galois equivariant exact sequence:
\begin{equation}
0\rightarrow \Z_p\({\Z_p(1)}\)\cdot y
\rightarrow \pi_1^\pol 
\xrightarrow{\proj'} \Z_p(1)\rightarrow 0.
\end{equation}
The sequence turns out to split, as the image 
$\mathrm{Im}(\proj')=\Z_p(1)$ 
can be lifted to a Galois-invariant inertia subgroup over $0$ 
along $\overrightarrow{01}$.
We shall take standard generators $x$, $y$ of those respective
inertia subgroups over $0$, $1$ so as to correspond to our choice
of $\{\zeta_{p^n}\}_n\in\mu_{p^\infty}$. 
It follows then that $\pi_1^\pol$ is isomorphic to 
a semi-direct product:
\begin{equation}
\label{semi-direct}
\pi_1^\pol
%x^{\Z_p}\ltimes \varprojlim_{n} \bigl(\Z_p[y]/(y^{p^n}-1)\bigr)
=x^{\Z_p}\ltimes \bigl(
\Z_p\(x^{\Z_p}\)\cdot y \bigr)
\cong \Z_p(1)\ltimes \bigl( \Z_p\( \Z_p(1)\)(1) \bigr).
\end{equation}
Here, the action in the last semi-direct product 
is given by the translation of $\Z_p(1)$. 
%Let us identify $\pi_1/[\Ker\proj,\Ker\proj]=\Z_p(1)\ltimes \Z_p\( \Z_p(1)\)(1)$.
Let $\Pet{}(\ovec{01})$ 
be the pro-unipotent completion of $\pi_1^\pol$
(which we call the {\it $p$-adic etale polylogarithmic group}), 
and let $\log: \Pet{}(\ovec{01}) \to \mathrm{Lie}(\Pet{}(\ovec{01}))
=\Q_p(1)\oplus \prod_{k=1}^\infty \Q_p(k)$ its logarithm map.
We have then the following Lie expansion map (also denoted $\log$):
\begin{equation}
\log: 
\pi_1^\pol
\hookrightarrow 
\Pet{}(\ovec{01})
\longrightarrow \mathrm{Lie}(\Pet{}(\ovec{01}))
=\Q_p(1)\oplus \prod_{k=1}^\infty \Q_p(k).
\end{equation}
In practice, both of $\Pet{}(\ovec{01})$ and 
$\mathrm{Lie}(\Pet{}(\ovec{01}))$ 
are realized as subsets of
the non-commutative power series ring 
$\Q_p\la\!\la X,Y\ra\!\ra$ in $X=\log(x)$, $Y=\log(y)$ modulo
the ideal $I_Y$ generated by the words having 
$Y$ twice or more.

Now, take any point $z\in(\P^1\setminus\{0,1,\infty\})(\overline{\Q_p})$
and consider the set $\pi_1(\ovec{01},z)$ of the pro-$p$ etale paths 
from $\overrightarrow{01}$ to $z$. Set $F:=\Q_p(z)$.
Then $\pi_1(\ovec{01},z)$ forms a $\pi_1$-torsor equipped with
a canonical action of $G_F$.
In the well-known manner, the reduction map
$\pi_1\ontomap \pi_1^\pol$ gives rise to a $\pi_1^\pol$-torsor 
$\pi_1^\pol(\ovec{01},z)$ with $G_F$-action on it.

Choose any path $\gamma\in \pi_1^\pol(\ovec{01},z)$ 
and call it the {\it defining path}. 
Then, for each $\sigma\in G_F$, 
there is a unique element $\ff_\sigma\in \pi_1^\pol$ such that 
$\gamma$ is written as $\ff_\sigma^z\ast \sigma(\gamma)$.
The $p$-adic Galois polylogarithms (associated to $\gamma$) 
are defined as coefficients of the expansion of the 
Lie series $\log(\ff_\sigma^z)$ in $\Q_p[[X,Y]]/I_Y$ as: 
\begin{equation}
\label{eq2.4}
\log(\ff_\sigma^z)\equiv
-\kappa_z(\sigma)X-\sum_{k=1}^\infty
\ell i_k(z)(\sigma)\cdot ad(X)^{k-1}(Y).
\end{equation}
The above first coefficient $\kappa_z:G_F\to \Z_p$ 
is the Kummer 1-cocycle 
\begin{equation}
\label{kummer}
\zeta_{p^n}^{\kappa_z(\sigma)}
=(z^{1/p^n})^{\sigma-1}:=
\frac{\sigma(z^{1/p^n})}{z^{1/p^n}}
\end{equation}
over the system 
$\{z^{1/p^n}\}_{n=1}^\infty$ 
which is determined by the specialization
homomorphism
$\bigcup_{n=1}^\infty \overline{F}[t^{1/p^n}]\to \overline{F}$ 
along the path $\gamma$.
The second coefficient $\ell i_1(z):G_F\to\Z_p$ 
is the Kummer 1-cocycle $\kappa_{1-z}$ 
along the composition of the standard path 
$0\nyoroto 1$ with $\bar\gamma:1\nyoroto 1-z$, 
where $\bar\gamma$ is the obvious reflection
of $\gamma$.
The other coefficients 
$\ell i_k(z) (=\ell i_k(z,\gamma))$ $(k\ge 2)$
are in general only 1-cochains $G_F\to \Q_p$.

The following lemma is crucial to understand nature of the LHS of
our main statement of Introduction.

\begin{lemma}[\cite{W1} Theorem 5.3.1] \label{lem2.1}
The fixed field $H_m$ of the intersection of the kernels of 
$\chi_\cyc$, $\kappa_z$ and of $\ell i_1(z,\gamma),\dots, \ell i_{m-1}(z,\gamma)$
is independent of the choice of the defining path 
$\gamma:\ovec{01}\nyoroto z$. Moreover, 
for $\sigma\in G_{H_m}$,  the value $\ell i_m(z,\gamma)\in\Q_p$
is independent of the choice of $\gamma$. \qed
\end{lemma}

Explicit formulas for $\ell i_k(z)$ have been given in
\cite{N-W1} for all $\sigma\in G_F$. For simplicity, we present the
formula for $\sigma\in G_F$ with $\chi_\cyc(\sigma)=1$ and $\kappa_z(\sigma)=0$
in the following 

\begin{proposition}[\cite{N-W1} \S 3 Corollary]
\label{prop2.2}
For $m\ge 1$ and $\sigma\in G_{F(\zeta_{p^\infty},z^{1/p^\infty})}$, 
we have 
\begin{equation}
\ell i_m(z)(\sigma)=(-1)^{m-1}\frac{\tilchi_m^z(\sigma)}{(m-1)!}
\end{equation}
where $\tilchi_m^z(\sigma)\in\Z_p$ 
is defined by the Kummer properties
\begin{equation}
\label{tilchiformula}
\zeta_{p^n}^{\tilchi_m^z(\sigma)}=
\biggl(\prod_{a=0}^{p^n-1}(1-\zeta_{p^n}^az^{1/p^n})^{\frac{a^{m-1}}{p^n}}
\biggr)^{\sigma-1}
\qquad (n\ge 1). \qed
\end{equation}
\end{proposition}

Indeed, in \cite{N-W1}, where an embedding 
$\overline{F}\hookrightarrow\C$ is fixed
and $\gamma:\vec{01}\nyoroto z$ is taken 
to be a continous curve on $\mathbb{P}^1(\C)-\{0,1,\infty\}$,
we constructed 
a measure ${\hat \bkappa}_{z,\gamma}(\sigma)\in\hat \Z\(\hat \Z\)$ 
called the (adelic) {\it Kummer-Heisenberg measure},
after specifying standard branches of power roots of 
$(1-\zeta_{n}^az^{1/n})$ along $\gamma$.
(In fact, in loc.cit., we assumed $z\in F$ a number field, 
but the argument goes similarly for any subfield $F$ of $\C$.)
The $\Z_p$-valued measure 
$\bkappa_{z,\gamma}\in\Z_p\(\Z_p\)$
obtained as the image of ${\hat \kappa}_{z,\gamma}(\sigma)$
is given by
\begin{equation}
\label{KummerHeisenberg}
\zeta_{p^r}^{{\kappa}_{z,\gamma}(\sigma)(a+p^n\Z_p)}
=
\frac{\sigma\bigl(
(1-\zeta_{p^n}^{\chi_\cyc(\sigma)^{-1}a}z^{1/p^n})^{\frac{1}{p^r}}
\bigr) }
{(1-\zeta_{p^n}^{a}\sigma(z^{1/p^n}))^{\frac{1}{p^r}}}
\qquad (a\in\Z_p, \ r\ge 1)
\end{equation}
for each 
$\sigma\in G_F$.
The {\it $p$-adic polylogarithmic character}
$\tilchi_m^z:G_F\to \Z_p$ 
defined in \cite{N-W1} can be written as the moment integral
\begin{equation}
\label{tilchidef}
\tilchi_m^z(\sigma)=\int_{\Z_p} x^{m-1} d\bkappa_{z,\gamma}(\sigma)(x)
\quad (\sigma\in G_F).
\end{equation}
In this paper, we shall also consider a restricted version of
the above moment integral % (\ref{tilchidef}) 
to $\Z_p^\times$,
and introduce the {\it restricted $p$-adic polylogarithmic character}
$\chi_m^z:G_{F}\to \Z_p$ by
\begin{equation}
\label{chidef}
\chi_m^z(\sigma)
=\int_{\Z_p^\times} x^{m-1} d\bkappa_{z,\gamma}(\sigma)(x)
\quad (\sigma\in G_{F}).
\end{equation}
Then, it is easy to see from Proposition \ref{prop2.2}
that the value $\chi_m^z(\sigma)\in\Z_p$ for 
$\sigma\in G_{F(\zeta_{p^\infty},z^{1/p^{\infty}})}$
is characterized by the Kummer properties
\begin{equation}
\label{chiformula}
\zeta_{p^n}^{\chi_m^z(\sigma)}=
\biggl(
\prod_{\substack{1\le a\le p^n \\ p\nmid a}}
(1-\zeta_{p^n}^az^{1/p^n})^{\frac{a^{m-1}}{p^n}}
\biggr)^{\sigma-1}
\qquad (n\ge 1).
\end{equation}
When $z=1$ and $F=\Q$, this is nothing but what is called 
the $m$-th Soule character.

\begin{remark} \label{Gabber}
%\label{rem2.3}
%\begin{itemize}
%\item[(1)]
Our considering the measure 
${\hat \bkappa}_{z,\gamma}(\sigma)$ should
be traced back partly to an old idea of O. Gabber
(as documented in \cite{N-W1} \S 3 Remark 1 and
\cite{N-W2} Acknowledgments). 
In \cite{W3}, it is generalized to a sequence of 
measures `$K_r(z)$' on $\Z_p^r$  
($r\geq 1$) that encodes all coefficients of $\log(\ff_\sigma)$ 
(which correspond to the multiple polylogarithms)
as integrals over $\Z_p^r$. 
\end{remark}

A simple connection between $\tilchi^z_m$ and $\chi^z_m$ can
be obtained by comparing 
(\ref{tilchiformula}) and (\ref{chiformula}).
The following generalizes \cite{N-W1} \S 2 Remark 2:

\begin{lemma}[\cite{W3} Proposition 5.1(v)]
\label{tilchi.and.chi}
For any continous path $\gamma:\ovec{01}\nyoroto z$
with respect to an embedding $\overline{F}\hookrightarrow \C$,
any $\sigma\in G_F$, %{F(\zeta_{p^\infty},z^{1/p^\infty})}$ 
and $m\ge 2$, we have 
\begin{equation*}\label{tilchi.chi}
\tilchi^z_m(\sigma)=\chi_m^z(\sigma)+p^{m-1}\chi_m^{z^{1/p}}(\sigma)+\cdots
=\sum_{k=0}^\infty p^{k(m-1)}\chi_m^{z^{1/p^k}}(\sigma).
\end{equation*}
\end{lemma}

\begin{proof}
Denote the path $\ovec{01}\nyoroto z^{1/p^k}$ induced from
$\gamma:\ovec{01}\nyoroto z$ by the same symbol $\gamma$. 
Then, it is easy to see from (\ref{KummerHeisenberg}) that
$\bkappa_{z,\gamma}(\sigma)(p^k a+p^{n+k}\Z_p)
=\bkappa_{z^{1/p^k},\gamma}(\sigma)(a+p^{n}\Z_p)
$ for all $k\ge 0$.
Then, decomposing the integral to infinite pieces as
\begin{align*}
\int_{\Z_p}x^{m-1} d\bkappa_{z,\gamma}(\sigma)(x)
&=\sum_{k=0}^\infty
\int_{p^k\Z_p-p^{k+1}\Z_p}x^{m-1} d\bkappa_{z,\gamma}(\sigma)(x) \\
&=\sum_{k=0}^\infty
\int_{\Z_p^\times} (p^kx)^{m-1} d\bkappa_{z^{1/p^k},\gamma}(\sigma)(x)
\end{align*}
immediately yields the desired formula.
%Note that the measure 
%$\bkappa_{z,\gamma}(\sigma)$ has no support on 
%$\{0\}\subset\Z_p$
%by our assumption 
%$\sigma\in G_{F(\zeta_{p^\infty},z^{1/p^\infty})}$,
%hence the decomposition holds even when $m=1$.
\end{proof}

\begin{remark} \label{remPoly}
%\item[(2)]
As noted above (after (\ref{kummer})), every etale path 
$\gamma:\ovec{01}\nyoroto z$
determines the value $z^{1/p^n}\in\overline{F}$,
hence the above Kummer quantity in 
 (\ref{KummerHeisenberg}) makes sense
for $\sigma\in G_{F(\zeta_{p^\infty},z^{1/p^\infty})}$
with no need to mention the choice of 
$\overline{F}\hookrightarrow \C$.
So, in what follows, we will consider both versions 
of $p$-adic polylogarithmic 
characters $\tilchi^z_m$, $\chi^z_m$
for arbitrary $\gamma\in\pi_1^\pol(\ovec{01},z)$ but 
only for $\sigma\in G_{F(\zeta_{p^\infty},z^{1/p^\infty})}$.
Formulas in Proposition \ref{prop2.2} and 
Lemma \ref{tilchi.and.chi} 
also hold for them.
\end{remark}

%Before closing this section, we add a simple
%
%\begin{lemma}
%\label{lemma2.3}
%If a measure $\mu$ on $\Z_p$ has a property that 
%$\mu(a+p^N\Z_p)=\mu(pa+p^{N+1}\Z_p)$ for all $a\in\Z_p$
%and $N\ge 1$, 
%then, for $m\in \Z_{>0}$,
%$$
%\int_{\Z_p}x^m d\mu(x)
%=
%\frac{1}{1-p^{m}}\int_{\Z_p^\times} x^m d\mu(x).
%$$
%\end{lemma}

%%%%%%%%%%%%%%%%%%%%%%%%%%%%%%%%%%%%%%%%%%%%%%%%%%%%%%%%%%%%%%%%%%%%%%%%%%

\subsection{Koblitz measure and $p$-adic polylogarithms}
Let $|\,|_p$ denote the standard norm on $\C_p$ with $|p|_p=p^{-1}$
and write $\cO:=\mathcal O_{\C_p}=\{z\in\C_p; |z|_p\le 1\}$. 
For $z\in \C_p$ with $|1-z|_p\geq 1$, 
Neal Koblitz introduced in \cite[p. 457]{Ko}, 
an $\cO$-valued measure 
$\mu_z$ on $\Z_p$ by 
\begin{eqnarray*}
\mu_z(a+p^n\Z_p)=\frac{z^a}{1-z^{p^n}}\ \text{for all }n\in\Z,\ 1\leq a\leq p^n.
\end{eqnarray*}
Note here that $|1-z^{p^n}|_p=|z^{p^n}|_p$ if $|z|_p>1$ and 
that $|1-z^{p^n}|_p\ge 1$ if $|z|_p\le 1$
for all $n\ge 0$ under our assumption on $z$.

\begin{lemma}
\label{eq2.10}
Let ${\mathscr F}_z(T)$ be the element of the Iwasawa algebra
$\cO\(T\)$ that corresponds to $\mu_z$, and let 
$\mu_z^{(p)}$ be the restriction of the measure
$\mu_z$ to $\Z_p^\times$.
Then, 
\begin{align*}
{\mathscr F}_z(T)&=\frac{1}{1-z(1+T)}\in \cO\(T\), 
\tag{1}\\
{\mathscr F}_z^{(p)}(T)&=\frac{1}{1-z(1+T)}-\frac{1}{1-z^p(1+T)^p}.
\tag{2}
\end{align*}
\end{lemma}

\begin{proof}
The first formula is a consequence of the congruence
$$
{\mathscr F}_z(T)=\frac{1}{1-z^{p^n}}
\left( \frac{1-z^{p^n}}{1-z(1+T)}
\right)
\equiv
\sum_{a=0}^{p^n-1}
\frac{z^a}{1-z^{p^n}}(1+T)^a
$$
modulo the ideal $((1+T)^{p^n}-1)$ in $\cO\(T\)$
for every $n$.
For the second one, recalling a formula \cite{C-S} \S 3.4 for 
the restriction of a measure on $\Z_p$ to $\Z_p^\times$,
we obtain
${\mathscr F}_z^{(p)}=
{\mathscr F}_z(T)-\frac{1}{p}\sum_{\xi\in\mu_p}{\mathscr F}_z(\xi(1+T)-1)$.
Apply then a general formula
$\sum_{\xi\in\mu_n}\frac{1}{1-\xi Y}=\frac{n}{1-Y^n}$.
\end{proof}

Note that there is an equality
\begin{equation}
\label{eq2.11}
{\mathscr F}_z^{(p)}(T)=
\sum_{n=0}^\infty
\left( \int_{\Z_p^\times}x^n d\mu_z(x)\right)
\frac{X^n}{n!}
\end{equation}
in $\C_p\(X\)$ with $1+T=e^X$.

In \cite{C5}, Coleman introduced the $p$-adic polylogarithm function
$\Li^\padic_k(z)$ ($k\ge 1$) and its companion function 
$\Li_k^{(p)}(z):=\Li^\padic_k(z)-p^{-k}\Li^\padic_k(z^p)$
(In his notation, $\Li^\padic_k(z):=\ell_k(z)$, 
$\Li_k^{(p)}(z):=\ell_k^{(p)}(z)$). 
The latter function is given by the Koblitz measure as follows:
\begin{eqnarray}
\label{ColemanIntegral}
\Li_k^{(p)}(z)=\int _{\Z_p^\times }x^{-k}d\mu_z(x)
=\int _{\Z_p^\times }x^{-k}d\mu^{(p)}_z(x)
\label{eq0.4}
\end{eqnarray}
(see \cite[Lemma 7.2]{C5}). 
In other words, if $|1-z|_p\geq 1$, we have the equality
\begin{eqnarray}
\Li_k^{(p)}(z)=
\lim_{n\rightarrow\infty}
\left(\sum_{a=1,\, p\nmid a}^{p^n}
\frac{z^a a^{-k}}{1-z^{p^n}}\right).
\label{eq0.6}
\end{eqnarray} 

\begin{remark}
It is worthwhile to recall the following:
Coleman \cite{C3} showed that $\Li^\padic_k(z)$ 
has ``analytic continuation along Frobenius'' to
$\mathbb{P}^1(\C_p)-\{1,\infty\}$
depending on branch parameter of Iwasawa logarithm
(cf.\,e.g.,\,\cite[p.425]{KN}). 
This result has been extended to the $p$-adic multiple polylogarithms
by H. Furusho \cite{Fu}.
See also Remark \ref{Gabber} for comparable progress in the theory
of Galois polylogarithms. 
\end{remark}

\medskip\noindent
{\bf Bernoulli distribution and Kubota-Leopoldt $L$-function}

Before closing this subsection, we quickly review classically 
known facts about the Bernoulli distribution for the case $z=1$. 
Let $c$ be an integer $\ne 1$ prime to $p$.
Then a measure $E_c$ on $\Z_p$ is defined as follows 
(cf. [Wa] \S 12.2):
For each $n\ge 1$, choose any integer $\bar c=\bar c_n$ 
with $c\bar c\equiv 1$ mod $p^n$, and define
\begin{equation}
E_c(a+p^n\Z_p):=B_1\left(\left\{
\frac{a}{p^n}\right\}\right)-
cB_1\left(\left\{\frac{\bar ca}{p^n}\right\}\right),
\end{equation}
where $\{\ast\}$ means the fractional part.
This is independent of the choice of $\bar c$, and forms
a $\Z_p$-valued measure $E_c$ on $\Z_p$.
A well-known calculation ($e^X=1+T$):
\begin{equation}
\label{eq2.15}
\frac{1}{T}-\frac{c}{(1+T)^c-1}
=\sum_{n=1}^\infty (1-c^n)\frac{B_n}{n}\cdot\frac{X^{n-1}}{(n-1)!}
\end{equation}
as well as the formula 
$L_p(1-n,\omega^n)=-(1-p^{n-1})\frac{B_n}{n}$
give
\begin{equation}
\label{eq2.16}
\int_{\Z_p^\times}x^{m-1}dE_c(x)
=(1-c^m)(1-p^{m-1})\frac{B_m}{m}
=(c^m-1)L_p(1-m,\omega^m).
\end{equation}
(Cf. proofs of [Wa] Cor.\,5.13, Cor.\,12.3).
A formula of the 
$p$-adic Mellin transformation reads then as follows:
\begin{equation}\label{eq2.17}
\int_{\Z_p^\times}
x^{-m}dE_c(x)
=
(c^{1-m}-1)L_p(m,\omega^{1-m}).
\end{equation}
(Cf. [Wa] Theorem 12.2 with $d=1$, $\chi=\omega^{1-m}$, $s=-m$.)
Finally, we remark that the power series 
${\mathscr F}_z^{(p)}$ of Lemma \ref{eq2.10} (2) 
for $z=1$ has a pole at $T=0$. However, its usual 
`$c$-correction' cancel the pole and has an expansion 
in the form:
\begin{equation}
\label{eq2.18}
{\mathscr F}_1^{(p)}(T)-c {\mathscr F}_1^{(p)}((1+T)^c-1)=
\sum_{n=0}^\infty
\left( \int_{\Z_p^\times}x^n dE_c(x)\right)
\frac{X^n}{n!}
\end{equation}
in $\Q_p\(X\)$ with $1+T=e^X$.
(Note LHS=$Ug(T)$ of [Wa] \S 12 p.251-252.)

%\newpage
\section{Special family of Coleman series}
%The case $z\in\muprime$}

\subsection{Basic setup and $f_{z,c}(T)$}
In this subsection, after reviewing some basic notions,  
we shall introduce a special class of Coleman power series
that play important roles in our proof of Theorem \ref{mainthm}.

Let $F$ be a finite unramified extension of $\Q_p$.
In this section, we study Galois and $p$-adic polylogarithms
$\ell i_m(z)$ and $\Li^\padic_m(z)$ for 
$z\in F\cap \muprime$, a root of unity of order prime to $p$.
We will introduce and observe behaviors of 
certain special power series 
$f_{z,c}(T)\in \cO_F\(T\)$ closely related to 
${\mathscr F}_z^{(p)}(T)$ of the previous section. 

We begin by setting up basic operations on power series. 
Set $F_\infty:=F(\mu_{p^\infty})$ and denote by $G_\infty$
the Galois group of $F_\infty/F$. 
Note that, the $p$-adic cyclotomic character induces the canonical 
isomorphism 
$$
\bar\chi_\cyc:G_\infty\xrightarrow{\sim}\Z_p^\times
$$ 
as $F$ is unramified over $\Q_p$. 
For each $a\in\Z_p^\times$, we define $\sigma_a\in G_\infty$ to be 
$\bar\chi_\cyc^{-1}(a)$. 
%Later, we also denote by $\sigma_a$ its image in $G_n$ for simplicity.

We set 
$\cO_F\(T\)^{\times 1}:=
\mathrm{Ker}(\cO_F\(T\)^{\times}
\xrightarrow{\mathrm{aug}}\mathcal O_F^\times 
\rightarrow k_F^\times)$. 
Here, the first map is the augmentation map and $k_F$ 
the residue field of $F$. 
Define the action of $G_\infty$ on $\cO_F\(T\)$ by
\begin{eqnarray*}
\sigma_a f(T)=f((1+T)^a-1).
\end{eqnarray*}
that restricts to the action on $\mathcal O_F\(T\)^{\times 1}$. 
Since $\mathcal O_F\(T\)$ and $\cO_F\(T\)^{\times 1}$ 
are compact (additive and multiplicative) $\Z_p$-modules , 
the complete group ring $\Z_p\(G_\infty\)$ acts on 
both topological abelian groups. 
We regard $\mathcal O_F\(T\)$ (resp. $\mathcal O_F\(T\)^{\times 1}$) 
as a left (resp. right) $\Z_p\(G_\infty\)$-module. 

\begin{remark}We remark the followings:
\begin{itemize}

\item[(1)] The $\Z_p$-module structure on 
$\mathcal O_F\(T\)^{\times 1}$ 
is given by
%\begin{eqnarray*}
$f(T)^c=\sum_{n=0}^\infty \binom{c}{n}\left(f(T)-1\right)^n$
for $c\in \Z_p$, $f\in \mathcal O_F\(T\)^{\times 1}$,
%\end{eqnarray*}
where $\binom{x}{n}:=\frac{x(x-1)\cdots (x-n+1)}{n!}\in \Z_p$.

\item[(2)]Since $\mathcal O_F\(T\)$ has a canonical structure 
of (additive) $\mathcal O_F$-module, the complete group ring 
$\mathcal O_F\(G_\infty\)$ acts on it. 
However, $\mathcal O_F\(T\)^{\times 1}$ does not have 
a canonical $\mathcal O_F$-module structure, i.e., there does not 
exist a canonical action of $\mathcal O_F\(G_\infty\)$ on it.

\item[(3)]
A finite element $\sum_{a\in\Z_p^\times} c_a \sigma_a$ 
($c_a=0$ for all but finitely many $a\in G_\infty$) acts
on $f(T)$ in the following forms:
$
\left(\sum_{a} c_a \sigma_a\right)\cdot f(T)
=\sum_{a} c_a f((1+T)^a-1),
$ 
and
$
f(T)^{\Sigma_{a} c_a \sigma_a}
=\prod_a f((1+T)^a-1)^{c_a}.
$
%the topological abelian group $\mathcal O_F\(T\)^{\times ,1}$.
\end{itemize}
\end{remark}

\noindent
{\bf Special Coleman series.}
For  $c\in \Z$ ($c\ne 1$, $p\nmid c$) and for $z\in F\cap \muprime$, 
let us
introduce power series $f_{z,c}\in\cO_F\(T\)$ by
\begin{equation}
f_{z,c}(T):=
\begin{cases}
\displaystyle
\quad c\cdot \frac{(1+T)^{-\frac{1}{2}}-(1+T)^{\frac{1}{2}}}
{(1+T)^{-\frac{c}{2}}-(1+T)^{\frac{c}{2}}}, &(z=1); \\[4mm]
%\quad & \\
\displaystyle
\quad %\prod_{i=0}^{d-1}
\frac{(1+T)^{-\frac{1}{2}}-z(1+T)^{\frac{1}{2}}}
{(1+T)^{-\frac{c}{2}}-z(1+T)^{\frac{c}{2}}}, 
&(z\in\muprime\setminus \{1\}).
\end{cases}
\end{equation}

We denote by $\sigma_F:F\rightarrow F$ the Frobenius automorphism of $F$. 
It acts on each element of $F\(T\)$ by the action on coefficients. 
Define the {\it integral logarithm}
$\cLog: \mathcal O_F\(T\)^\times
%\mathcal O_F\(T\)^{\times 1}
\rightarrow\mathcal O_F\(T\)$ 
by the formula
\begin{eqnarray*}
\cLog(f(T)):=\frac{1}{p}\log\left(\frac{f(T)^p}{(\sigma_F f)((1+T)^p-1)}\right).
\end{eqnarray*}
Further, we define the differential operator 
$D$ on $F\(T\)$ to be $(1+T)\frac{d}{dT}$ and
the ring homomorphism 
$[p]:\mathcal O_F\(T\)\rightarrow \mathcal O_F\(T\)$ 
by $([p]f)(T)=f((1+T)^{p}-1)$. 

Let $\Norm:\mathcal O_F(\!(T)\!)^\times\rightarrow 
\mathcal O_F(\!(T)\!)^\times$ 
(resp.   $\bTrace:\mathcal O_F\(T\)\rightarrow \mathcal O_F\(T\)$)
be the norm operator (resp. the `reduced' trace operator) 
introduced by Coleman (cf.\,\cite{C1}, \cite[p.386]{C2}). 
They are characterized by
\begin{equation*}
([p]\Norm f)(T)=\prod_{\xi\in\mu_p} f(\xi (1+T)-1), \quad
([p]\bTrace f)(T)=\frac{1}{p}\sum_{\xi\in\mu_p} f(\xi (1+T)-1).
\end{equation*}

\begin{lemma} \label{lem3.1}
Let $F$ be a finite unramified extension of $\Q_p$
containing $z\in\muprime$. Then,
\begin{align*}
&f_{z,c}(T)\in 1+T \mathcal O_F\(T\). \tag{1}
\\
&\Norm(f_{z,c}(T))=(\sigma_Ff_{z,c})(T). 
\tag{2}
\\
&D{\cLog}f_{z,c}(T)
=
-\left({\mathscr F}_z^{(p)}(T)
-c {\mathscr F}_z^{(p)}((1+T)^c-1)
\right)
%\begin{cases}
%\displaystyle\sum_{m=0}^\infty (c^{m+1}-c)
%\left( \int_{\Z_p^\times}x^m dE_c(x)\right)
%\frac{X^m}{m!}  &(z=1); \\
%\displaystyle\sum_{m=0}^\infty (c^{m+1}-c)
%\left( \int_{\Z_p^\times}x^m d\mu_z(x)\right)
%\frac{X^m}{m!}  &(z\ne 1). 
%\end{cases}
\tag{3}
\end{align*}
\end{lemma}

\begin{proof}
(1) This is just claiming $f_{z,c}(0)=1$. 
Use d'H{\^o}spital rule when $z=1$.
(2) 
%Since the last two assertion follows from (1) of the lemma, 
%we show (1) of Lemma \ref{lem2.1}. 
For $a\in \Z_p^\times$, set 
$k_{z,a}(T):=(1+T)^{-\frac{a}{2}}-z(1+T)^{\frac{a}{2}}$. 
Then, by the definition of the norm operator, we have:
\begin{eqnarray*}
([p]\mathcal Nk_{z,a})(T)
&=&\prod_{\xi\in \mu_p}
\left(
(\xi(1+T))^{-\frac{a}{2}}-z(\xi(1+T))^{\frac{a}{2}}\right)\\
&=&\prod_{\xi\in \mu_p}\left((1+T)^{-\frac{a}{2}}
-z\xi^a(1+T)^{\frac{a}{2}}\right) \\
&=& (1+T)^{-\frac{pa}{2}}-z^p(1+T)^{\frac{pa}{2}}
=([p]k_{z^p})(T). 
\end{eqnarray*}
As $[p]$ is injective, 
it follows that $\mathcal Nk_{z,a}=k_{z^p,a}$. 
Since $\sigma_F$ acts on $\muprime$ by $p$-power 
and the operator $\mathcal N$ is multiplicative, 
the assertion (2) follows.
%$\mathcal N(G_z(T))=G_{z^p}(T)$.
(3) also follows from a simple calculation
\begin{equation*}
D\cL(k_{z,a}(T))=
a\left(
-\frac{1}{1-z(1+T)^a}+\frac{1}{1-\sigma_F(z)(1+T)^{ap}}
\right)
\end{equation*}
with Lemma \ref{eq2.10} (2), including the case $z=1$ where
${\mathscr F}_1^{(p)}(T)-c{\mathscr F}_1^{(p)}((1+T)^c-1)\in\Z_p\(T\)$
is discussed in (\ref{eq2.18}).
(See also Remark \ref{rem1.3} below.)
\end{proof}

\begin{remark} 
\begin{itemize}

\item[(1)]
The integral logarithm $\cLog$ on $\cO_F\(T\)^{\times 1}$
preserves the action of $\Z_p\(G_\infty\)$-action, namely, it holds
that 
$\cLog(f^\lambda)=\lambda\cdot\cLog(f)$ 
for any $\lambda\in\Z_p\(G_\infty\)$ and 
$f\in \mathcal O_F\(T\)^{\times 1}$.
By a simple argument, one can easily see that 
$\Norm(f)=\sigma_F(f)$ implies $\bTrace(\cLog f)=0$. Namely,
$\cLog$ maps $(\cO_F\(T\)^\times)^{\Norm=\sigma_F}$ into 
$\cO_F\(T\)^{\bTrace=0}$. 

%$\check$
%There exists an exact sequence of $\Z_p\(T\)$-modules:
%%\begin{equation*}
%0
%\rightarrow 
%%\mu_F\times (1+T)^{\Z_p}\xrightarrow{\iota} 
%\mathcal O_F\(T\)^{\mathcal N=\mathrm{\sigma_F}}
%\xrightarrow{\cLog}\mathcal O_F\(T\)^{\bTrace=0}
%\xrightarrow{\alpha}\mathcal O_F\rightarrow 0,
%\label{seq3.0}
%\tag{$\check$}
%\end{equation*}
%where $\mu_F$ is the set of roots of unity in $F$, 
%$\iota$ the canonical inclusion and $\alpha(f):=Df(T)|_{T=0}$.

\item[(2)]
If $f\in\cO_F\(T\)^\times$ is not necessarily 
in $\cO_F\(T\)^{\times 1}$, then
$\log(f)$ has no obvious sense and the convergence 
(and stay) of $\cLog(f)$ in $\mathcal O_F\(T\)$ 
involves technical estimate of coefficients 
(see \cite{C-S} Lemma 2.5.1).
We have $\cLog(fg)=\cLog(f)+\cLog(g)$ as long as both
$f,g\in \cO_F\(T\)^{\times}$. 
But, in practical computation, it could happen 
that we know the existence of $\cLog(fg)$ but 
not of individual $\cLog(f)$ or $\cLog(g)$.
A possible way to remedy such a computational difficulty
(which occurs also in the proof of
the above Lemma \ref{lem3.1} (3) when $z=1$) 
is to consider 
the logarithmic derivative $D\cLog(f)=(1-[p]\sigma_F)\cdot(Df/f)$ 
(whose existence is often easier to see) 
and to use $D\cLog(fg)=D\cLog(f)+D\cLog(g)$.

\item[(3)]
As remarked in \cite{C4}, Corollary of Theorem 3, 
the differential operator $D$ gives an isomorphism 
$D: \cO_F\(T\)^{\bTrace=0}\isom \cO_F\(T\)^{\bTrace=0}$.
(The proof given in loc.\,cit.\,for $\Z_p \(T\)$ works also for
$\cO_F\(T\)$ with no obstructions.)
This defines the inverse map
$$
D^{-1}: \cO_F\(T\)^{\bTrace=0}\isom \cO_F\(T\)^{\bTrace=0}.
$$
\end{itemize}
\label{rem1.3}
\end{remark}

%%%%%%%%%%%%%%%%%%%%%%%%%%%%%%%%%%%%%%%%%5
\begin{lemma}
\label{lem3.4}
For a power series $f(T)\in\cO_F\(T\)$, denote 
by $\mu_f$ the corresponding 
$\cO_F$-valued measure on $\Z_p$.
Suppose $f(T)\in\cO_F\(T\)^{\bTrace=0}$.
% where $\cO_F$ is the
%ring of integers of a finite extension of $\Q_p$. 
Then, for $k\in\Z$, it holds that
$$
D^{-k}(f)(T)=\sum_{n=0}^\infty
\left(
\int_{\Z_p^\times}x^{n-k}d\mu_f(x)
\right) \frac{X^n}{n!}
$$
in $F\(X\)$ with $e^X=1+T$.
In particular, we have
$$
D^{-k}(f)(0)=
\int_{\Z_p^\times}x^{-k}d\mu_f(x).
$$
\end{lemma}

\begin{proof}
The case of $k\le 0$ is well known (e.g.,
\cite{C-S} Lemma 3.3.5). So, assume $k>0$.
Consider a linear functional $L$ on the 
$\cO_F$-valued continuous functions on $\Z_p$
defined by $h\mapsto \int_{\Z_p} h x^{-k} d\mu_f(x)$.
This is well defined, as $\bTrace(f)=0$ insures
the support of the measure $\mu_f$ is on $\Z_p^\times$.
Since $L$ is bounded with support in $\Z_p^\times$, 
it follows that $L$ is of the form 
$h\mapsto \int_{\Z_p} h\, d\mu_g(x)$
for a unique $g(T)\in\cO_F\(T\)^{\bTrace=0}$
that, by the construction, should have the same expansion 
as RHS of the lemma (cf.\,\cite{C-S} p.35 and Lemma 3.3.5).
Since $D=(1+T)\frac{d}{dT}=\frac{d}{dX}$, we find
$D^{k}(g)=f$. The uniqueness of $g$ (cf. \cite{C4}
Corollary of Theorem 3) in $\cO_F\(T\)^{\bTrace=0}$
concludes $D^{-k}(f)=g$.
\end{proof}

\subsection{Coleman's reciprocity law}
\label{sec3.2}

\begin{definition}[Coates-Wiles homomorphism]
\label{DefCW}
Let $\mathcal U_\infty(F)$ be the norm limit of principal units 
of $\{F(\mu_{p^n})\}_n$ and denote by 
$$
[\mathrm{Col}]:\mathcal U_\infty(F)\rightarrow
\cO_F\(T\)^{\times 1}\cap (\cO_F\(T\)^\times)^{\Norm=\sigma_F}, \ 
\epsilon=(\epsilon_n)\mapsto g_{\epsilon}(T)
$$ 
the Coleman map which is characterized by 
\begin{equation}
\label{ColemanSeries}
(\sigma_F^{-n}g_\epsilon)(T)|_{T=\zeta_{p^n}-1}
=\epsilon _n.
\end{equation} 
The Coates-Wiles homomorphism 
$
\phi^{CW}_{m,F}:\mathcal U_\infty(F)\rightarrow \mathcal O_F
$
is defined by the equality
\begin{equation} \label{defCW}
\log(g_\epsilon(T))=\sum_{m=0}^\infty
\frac{ \phi^{CW}_{m,F}(\epsilon)}{m!}X^m
\end{equation}
in $F\(X\)$ with $1+T=\exp(X)$.
%$\check$ Is it true $g_\epsilon(0)=1$, so $\sum_{m\ge 1}$ above?
\end{definition}

%As in {\it Notation} of Introduction, we write $G_n$ for 
%the Galois group $\Gal(F(\mu_{p^n})/F)$ ($1\le n\le \infty$). 
Let
$
\rec_n:F(\mu_{p^n})^\times\rightarrow 
G_{F(\mu_{p^n})}^\mathrm{ab}
%\mathrm{Gal}(\Omega_p(H(\mu_{p^n}))/H(\mu_{p^n}))
$
be the reciprocity map of local class field theory
($n\le\infty$). 
When $n=\infty$, the reciprocity map induces an embedding 
$\rec_\infty:\cU_\infty(F)\hookrightarrow G_{F_\infty}^\ab$
(recall $F_\infty=F(\mu_{p^\infty})$, as defined 
in Notation of Introduction).
The above Coates-Wiles homomorphism $\phi^{CW}_{m,F}$ extends
uniquely to a $G_\infty=\Gal(F_\infty/F)$-homomorphism 
from $G_{F_\infty}^\ab$ into $F(m):=F\otimes\Z_p(m)$.
This extension and its image in
$\Hom_{G_\infty}(G_{F_\infty}^\ab, F(m))\cong H^1(F,F(m))$
will also be denoted by $\phi^{CW}_{m,F}$
 (cf.\, Bloch-Kato \cite{B-K} Section 2).

\begin{remark}
In (\ref{ColemanSeries}), we employ Bloch-Kato's normalization 
(\cite{B-K}, Theorem 2.2) on 
powers of $\sigma_F$, which differs from that in \cite{C2}.
The constant term of $g_\epsilon$ is $\equiv 1$ mod $p$, but may not be $1$.
This causes our summation in (\ref{defCW}) to start from $m=0$ which 
modifies \cite{B-K} (p.344).
\end{remark}

The Hilbert norm residue symbol
$$
(\ ,\ )_{p^n}:F(\mu_{p^n})^\times \times 
F(\mu_{p^n})^\times\rightarrow \mu_{p^n}
$$
is defined by the formula
$$
(a,b)_{p^n}=\bigl(a^{1/p^n}\bigr)^{\rec_n(b)-1}.
%(a^{\frac{1}{p^n}})a^{\frac{-1}{p^n}}.
$$
We shall make use of Coleman's explicit reciprocity law on
Hilbert norm residue symbols in the following form:
Recall that Coleman \cite{C2} introduces a continuous linear functional
$\int_n: F(\!(T)\!)_1\to F$ (where $F(\!(T)\!)_1$ denotes the
ring of power series which converge on the unit open ball on $\C_p$)
by 
$$
\int_n f:=\frac{1}{p^n}\sum_{\zeta\in\mu_{p^n}}f(\zeta-1).
$$

\begin{theorem}[Coleman \cite{C2},\cite{C5}] \label{ColemanFormula}
Let $f(T)\in1+T\cO_F\(T\)$ satisfy $\Norm(f)=\sigma_F(f)$, and
$g=g_\epsilon(T)=[\mathrm{Col}](\epsilon)$ be
the Coleman power series associated to $\epsilon=(\epsilon_n)\in
\mathcal U_\infty$.
Then, 
$$
(f(\zeta_{p^n}-1), \epsilon_n)_{p^n}
=\zeta_{p^n}^{\mathrm{Tr}_{F/\Q_p}
\left(\int_n \cLog(f)\cdot D\cLog(\sigma_F^{-n} g_\epsilon)\right)
}
$$
\end{theorem}

\begin{proof}
For reader's convenience, we shall show how to derive this formula from
Coleman's work: A direct application of the formula in \cite{C2} Theorem 1
tells that the exponent of $\zeta_{p^n}$ in RHS is 
$$
\mathrm{Tr}_{F/\Q_p}
\bigl(\int_n \cLog(f)\cdot D\log(\sigma_F^{-n} g_\epsilon)\bigr).
$$
We remark that there exists no error term in the sense of Coleman
(cf. loc.cit.) as $\Norm(g_{\epsilon})=\sigma_F(g_{\epsilon})$
(hence $k(0)=0$ in his notation).
Since $\frac{Dg}{g}-D\cLog(g)=[p]\sigma_F(\frac{Dg}{g})$
(Remark \ref{rem1.3} (2)), it suffices to show
$\int_n \cLog(f)\cdot [p](\sigma_F\frac{Dg}{g})=0$. This follows from
\cite{C5} (4.4), as $\bTrace(\cLog(f))=0$ when $\Norm(f)=\sigma_F(f)$
(cf. Remark \ref{rem1.3} (1)).
\end{proof}

%%%%%%%%%%%%%%%%%%%%%%%%%%%%%%%%%%%%%%%%%%%%%%%%%
%%%%%%%%%%%%%%%%%%%%%%%%%%%%%%%%%%%%%%%%%%%%%%%%%%%%%%%
\section{Proof of Main formula}
%Case of $z\in \muprime$ tailed with specific path}
In this section, we fix a root of unity $z$ in $F\cap \muprime$
(i.e., of order prime to $p$).
Note then that $F(\mu_{p^\infty},z^{1/p^\infty})=F_\infty$,
so that the $p$-adic polylogarithmic characters 
$\tilchi^z_m(\sigma)$, $\chi^z_m(\sigma)$ are defined
for $\sigma\in G_{F_\infty}$ and for arbitrary etale paths
$\gamma\in\pi_1^\pol(\ovec{01},z)$ as in Remark \ref{remPoly}.
Since the mod $p$ reduction gives an isomorphism 
$\muprime\isom \overline{\mathbb{F}}_{p}^\times$,
we can pick an integer $d\ge 1$ such that $z^{p^d}=z$.
%Note then that $z$ is a primitive $(p^d-1)$-th root of unity. 

\begin{lemma}
\label{specific_path}
There is an etale path $\gamma:\vec{01}\nyoroto z$ on
$\mathbb{P}^1_{\overline{\Q}_p}-\{0,1,\infty\}$
that determines a compatible branch $z^{1/p^n}$ {\rm ($n=1,2,\dots$)}
inside $\muprime$ with period $d$, i.e., $z^{1/p^n}=z^{1/p^{n+d}}$.
\end{lemma}
\begin{proof}
If one changes the choice of $\gamma$ by composition with
$x^b$ ($b\in \Z_p$), then the induced branch changes
from $z^{1/p^n}$ to $z^{1/p^n}\zeta_{p^n}^b$. 
As there is only one element in
$(z^{1/p^n}\cdot\mu_{p^n})\cap \muprime$ under the assumption,
one may find a correct class $b$ (mod $p^n$) 
for each $n\ge 1$, and hence get a correct 
$b$ as their limit.
\end{proof}

We define the ``weight accelerator'' homomorphism 
$$
\mathcal O_F\(G_\infty\)\rightarrow \mathcal O_F\(G_\infty\);
\quad \omega\mapsto \omega(k)
$$
for $k\in\Z$ to be the $\cO_F$-linear extension of
the mapping $\sigma_a\mapsto a^k\sigma_a$ ($a\in \Z_p^\times$).
It is not difficult to see that, for every 
$\omega\in \cO_F\(G_\infty\)$,
\begin{equation} 
\label{eq.3.3}
D\cdot \omega=\omega(1)\cdot D  
\end{equation}
holds as operators on $\mathcal O_F\(T\)$.
Let us also introduce the basic element
$$
\omega_n:=%\sum_{\substack{1\le i\le p^n \\p\nmid i}}
\sum_{{1\le i\le p^n, \, p\nmid i}}
\sigma_i\in\mathcal \Z_p\(G_\infty\).
$$

\begin{proposition}
Let $z\in \muprime$ and let 
$z^{1/p^n}$ {\rm ($n=1,2,\dots$)} be the compatible sequence
taken inside $\muprime$.
Let $F$ be a finite unramified extension 
of $\Q_p$ containing $z$ and hence all of $z^{1/p^n}$. 
Then, for $\epsilon=(\epsilon_n) \in \mathcal U_\infty(F)$ 
and for any positive integer $m$, we have:
\begin{equation*}
\left(
\bigl((f_{z^{1/p^n},c})^{\omega_n(m-1)}\bigr)(\zeta_{p^n}-1),
%f_0^+(\xi_{p^n}-1)^{\theta^+_{\zeta,m,c}},
\epsilon_n\right)_{p^n} \\
=
\zeta_{p^n}^{(-1)^{m-1}(c^{1-m}-1)
\mathrm{Tr}_{F/\Q_p}\bigl(
%(%\sigma_F^{-1}
\Li_m^{(p)}(z)
%)
\cdot
(1-p^{m-1}\sigma_F)\phi^{CW}_{m,F}(\epsilon)
\bigr).
}
\end{equation*}
Here, if $z=1$, we understand 
$\Li_m^{(p)}(1)$ represents $L_p(m,\omega^{1-m})$. 
\label{prop3.1}
\end{proposition}

\begin{proof}
We denote the left hand side by $\zeta_{p^n}^\alpha$. 
Then, according to Theorem \ref{ColemanFormula} 
and \cite{C5} Lemma (4.6),
%\cite[Theorem 1]{C2} (cf. \cite[(4.7.1)]{C5}), 
we have
\begin{align*}
\alpha
&\equiv \mathrm{Tr}_{F/\Q_p}
\left(
\int_n 
{\omega_n(m-1)}\cLog( f_{z^{1/p^n},c}(T))
\cdot
D\cLog(\sigma_F^{-n}g_\epsilon) 
\right) \\
&\equiv
\mathrm{Tr}_{F/\Q_p}
\left(
(-1)^{m-1}\bigl(D^{-(m-1)}\cLog f_{z^{1/p^n},c}\bigr)(0)
\cdot \bigl(D^m\cLog (\sigma_F^{-n}g_\epsilon) \bigr)(0)
%\frac{Dg_{\epsilon}(T)}{g_\epsilon(T)}\
\right)
\qquad  
(\mathrm{mod}\ p^n).
\end{align*}
%where $f(T):=$.
%and $\int_nh(T):=\frac{1}{p^n}\bTrace^n(h(T))|_{T=0}$. 
Observe from (\ref{defCW}) that 
\begin{equation*}
\bigl(D^m\cLog (\sigma_F^{-n}g_\epsilon) \bigr)(0)
=D^m\cLog (\sigma_F^{-n}g_\epsilon) |_{T=0}
=\sigma_F^{-n}(1-p^{m-1}\sigma_F)\phi^{CW}_{m,F}(\epsilon).
\end{equation*} 
On the other hand, 
noting that $(1-c\sigma_c)(-m)=1-c^{1-m}\sigma_c $ 
and hence 
$D^{-m}(1-c\sigma_c)=(1-c^{1-m}\sigma_c )D^{-m}$ 
by (\ref{eq.3.3}),
we obtain for $z\ne 1$,
\begin{align*}
D^{-m}D\cLog\left(f_{z^{1/p^n},c}(T)\right) 
&=D^{-m}(1-c\sigma_c)(-{\mathscr F}_{z^{1/p^n}}^{(p)}(T)) 
\quad(\text{Lemma \ref{lem3.1} (3)})
\\
&=-(1-c^{1-m}\sigma_c)D^{-m}({\mathscr F}_{z^{1/p^n}}^{(p)}(T)) 
\\
&=-(1-c^{1-m}\sigma_c)\sum_{n=0}^\infty
\left(
\int_{\Z_p^\times} x^{n-m} d\mu_{z^{1/p^n}}(x)
\right)
\frac{X^n}{n!},
\end{align*}
where Lemma \ref{lem3.4} is used in the last equality.
Hence 
$$
\bigl(D^{-m+1}\cLog f_{z^{1/p^n},c}\bigr)(0)=
(c^{1-m}-1)\Li^{(p)}_m(z^{1/p^n})=
(c^{1-m}-1)\bigl(\sigma_F^{-n}\Li^{(p)}_m(z)\bigr)
$$ 
according to our choice of $z^{1/p^n}$.
Therefore 
\begin{align*}
\alpha
&\equiv (c^{1-m}-1)
\mathrm{Tr}_{F/\Q_p} \left(\sigma_F^{-n}\Li^{(p)}_m(z)\cdot
\sigma_F^{-n}(1-p^{m-1}\sigma_F)\phi^{CW}_{m,F}(\epsilon)
\right) \\
&= (c^{1-m}-1)
\mathrm{Tr}_{F/\Q_p} \left(
%\sigma_F^{-1}
\Li^{(p)}_m(z)\cdot
(1-p^{m-1}\sigma_F)\phi^{CW}_{m,F}(\epsilon)
\right)
\end{align*}
as desired.
For $z=1$, we obtain similarly
\begin{align*}
D^{-m}D\cLog\left( f_{1,c}(T) \right) 
&=D^{-m}(1-c\sigma_c)(-{\mathscr F}_1^{(p)}(T)) 
\quad(\text{Lemma \ref{lem3.1} (3)})
\\
&=D^{-m}\sum_{n=0}^\infty 
\left(\int_{\Z_p^\times} x^{n} dE_c(x) \right)
\frac{X^n}{n!} 
\quad(\ref{eq2.18})
\\
&=\sum_{n=0}^\infty
\left(\int_{\Z_p^\times} x^{n-m} dE_c(x)\right)
\frac{X^n}{n!}
\quad(\text{Lemma \ref{lem3.4}}),
\end{align*}
hence, by (\ref{eq2.17}),  
$\bigl(D^{-m+1}\cLog f_{1,c}\bigr)(0)=
(c^{1-m}-1)L_p(m,\omega^{1-m})$.
This completes the proof of the proposition.
\end{proof}

\begin{lemma}
\label{lem3.3}
Let $m$ be a fixed positive integer. 
Then, there is an integer $N_m$ such that
for every $n\ge 1$, $p^n$ divides  
%$\displaystyle{
$N_m\sum_{\substack{1\le a\le p^n \\ p\nmid a}} a^m$.
%}$.
\end{lemma}

\begin{proof}
In fact, from the classical Bernoulli formula of power sums, it follows that
$$
\sum_{\substack{1\le a\le p^n \\ p\nmid a}} a^m
=
\frac{p^n}{m+1}\sum_{k=0}^{m}\binom{m+1}{k} B_k\cdot
(p^{n(m-k)}-p^{(n-1)(m-k)+m}).
$$
Thus, any common multiple of $m+1$ and of the denominators of 
Bernoulli numbers $B_0,\dots, B_m$ will do the role of $N_m$. 
\end{proof}

%\newpage
Now, we shall show our main formula:

\begin{theorem}
\label{fullformula}
For $z\in \mu(\Z_p^\ur)$, 
let $\gamma:\ovec{01}\nyoroto z\in\muprime$ 
be a specific path of Lemma \ref{specific_path}, and let
$F$ be a finite unramified extension of $\Q_p$ containing $z$.
Suppose that $\sigma\in G_F$ lies in the image of
$%\mathcal U_{\infty}(\Q_p)\to 
\mathcal U_\infty(F)\overset{\rec_\infty}{\longrightarrow}G_{F_\infty}^\ab$.
Then, for $m\ge 1$,
\begin{equation*} 
\chi_m^z(\sigma) =
{(-1)^{m}\,
\mathrm{Tr}_{F/\Q_p}\bigl(
\Li_m^{(p)}(z)\cdot
(1-p^{m-1}\sigma_F)\phi^{CW}_{m,F}(\sigma)
\bigr).
}
\tag{1}
\end{equation*}
Moreover, if $m\ge 2$, then,
\begin{equation*}
\tilchi^z_m(\sigma) =
%\frac{(-1)^{m}}{1-p^{(m-1)d}}
%\sum_{k=0}^{d-1}
%p^{(m-1)k}\ 
%\mathrm{Tr}_{F/\Q_p}\bigl(
%\Li_m^{(p)}(z^{1/p^k})\cdot
%(1-p^{m-1}\sigma_F)\phi^{CW}_{m,F}(\sigma)
%\bigr).
(-1)^{m}
\mathrm{Tr}_{F/\Q_p}
\bigl(Li_m^{(p)}(z)\phi^{CW}_{m,F}(\sigma)
\bigr).
\tag{2}
\end{equation*}
Here, if $z=1$, we understand 
$\Li_m^{(p)}(1)=L_p(m,\omega^{1-m})$. 
\label{prop4.4}
\end{theorem}

\begin{proof}
(1)
Choose $c\in\Z$ so that $p\nmid c$ and $c\ne 1$.
When $z=1$, we moreover impose $c\in (1+p\Z_p)^{N_m}$
for some $N_m$ as in Lemma \ref{lem3.3}. 
For $n\ge 1$, pick an integer 
$\bar c=\bar c_n\in \Z$ such that
$c\bar{c}\equiv 1$ mod $p^n$.
Then, by simple computation, it follows that
$$
\bigl((f_{z^{1/p^n},c})^{\omega_n(m-1)}\bigr)(\zeta_{p^n}-1)
\equiv
\xi\cdot
\prod_{\substack{1\le a\le p^n \\ p\nmid a}}
(1-z^{1/p^n}\zeta_{p^n}^a)^{a^{m-1}(1-\bar{c}^{m-1}})
$$
modulo $F(\zeta_{p^n})^{\times p^n}$
for some constant $\xi$ that lies in
$\mu_{p^n}\cdot (1+p\Z_p)^{p^n}$.
Putting this into Proposition \ref{prop3.1}, we obtain
\begin{equation} \label{HojoGodo}
(1-\bar c^{m-1})\chi_m^z(\sigma)\equiv
{(-1)^{m-1}(c^{1-m}-1)
\mathrm{Tr}_{F/\Q_p}\bigl(
\Li_m^{(p)}(z)\cdot
(1-p^{m-1}\sigma_F)\phi^{CW}_{m,F}(\sigma)
\bigr)
}
\end{equation}
mod $p^n$, 
where, if $z=1$, $\Li_m^{(p)}(z)$ stands for $L_p(m,\omega^{1-m})$.
Replacing $(1-\bar c^{m-1})$ by $(1-c^{1-m})$ 
in LHS of (\ref{HojoGodo}),
and then passing over $n\to \infty$, we obtain
the formula of proposition.

(2) Suppose $z^{p^d}=z$ for $d\in\Z_{>0}$.
Then, our specific choice of $\gamma$ makes
the sequence $\{z^{1/p^k}\}_{k\in\Z}$ cyclic of period $d$.
By using Lemma \ref{tilchi.and.chi}, 
we can relate $\tilchi_m^z(\sigma)$ to $\chi_m^z(\sigma)$ 
in such a way that $\tilchi_{m}^z(\sigma)
=\sum_{k=0}^{d-1}\frac{p^{(m-1)k}}{1-p^{(m-1)d}}
\chi_{m}^{z^{1/p^k}}(\sigma)
$
(cf.\,also \cite{W3} Proposition 5.3 (i)).
Putting the formula (1) into this, 
and replacing 
$\Li_m^{(p)}(z^{1/p^k})\cdot\sigma_F\bigl(
\phi_{m,F}^{CW}(\sigma)\bigr)$
by $\sigma_F\bigl(\Li_m^{(p)}(z^{1/p^{k+1}})\cdot
\phi_{m,F}^{CW}(\sigma)\bigr)$,
we may rewrite $\tilchi_{m}^z(\sigma)$ 
with 
$\alpha_k:=p^{(m-1)k}
%\mathrm{Tr}_{F/\Q_p}
\bigl(\Li_m^{(p)}(z^{1/p^k})\cdot\phi_{m,F}^{CW}(\sigma)
\bigr)$ as follows:
\begin{align*}
\tilchi_m^z(\sigma)&=
\frac{(-1)^m}{1-p^{(m-1)d}}
\sum_{k=0}^{d-1}
\mathrm{Tr}_{F/\Q_p}
\bigl(\alpha_k-\sigma_F(\alpha_{k+1})\bigr)
\\
&=
\frac{(-1)^m}{1-p^{(m-1)d}}
\sum_{k=0}^{d-1}
\bigl(
\mathrm{Tr}_{F/\Q_p}(\alpha_k)-
\mathrm{Tr}_{F/\Q_p}(\alpha_{k+1}) 
\bigr)
\\
&=\frac{(-1)^m}{1-p^{(m-1)d}}
\mathrm{Tr}_{F/\Q_p}(\alpha_0-\alpha_d).
\end{align*}
But since $z^{1/p^d}=z$, we have $\alpha_d=p^{(m-1)d}\alpha_0$.
This enables us to cancel the denominator $(1-p^{(m-1)d})^{-1}$ 
in the above expression,
and hence to conclude $\tilchi_m^z(\sigma)=(-1)^m\mathrm{Tr}_{F/\Q_p}(\alpha_0)$ 
as asserted in (2).   
\end{proof}

\begin{remark}
\label{finalremark}
The original Coleman-Ihara formula (\ref{FCW}) in Introduction
results from the above 
Theorem \ref{fullformula} (1) with $z=1$ and $F=\Q_p$. 
Note $(-1)^m=-1$ for odd $m$.
\end{remark}

We are arriving at the main Theorem \ref{mainthm} 
of Introduction: 

\begin{corollary} 
\label{maincor}
Let $\gamma:\ovec{01}\nyoroto z\in\muprime$ 
be a specific path of Lemma \ref{specific_path}, and let
$F$ be a finite unramified extension of $\Q_p$ containing $z$.
Suppose that $\sigma\in G_F$ lies in the image of
%$\mathcal U_{\infty}(\Q_p)\to 
$\mathcal U_\infty(F)\overset{\rec_\infty}{\longrightarrow}G_{F_\infty}^\ab$.
Then, for $m\ge 2$, 
\begin{equation*}
%\int_{\Z_p}x^{m-1} d\kappa_{z,\gamma}(\sigma)=
\ell i_m(z,\gamma)(\sigma)
=\frac{-1}{(m-1)!}
\mathrm{Tr}_{F/\Q_p}\left( 
\left\{
\left(
1-\frac{\sigma_F}{p^m}
\right)
\Li_m^\padic(z) 
\right\}
\phi_{m,F}^{CW}(\sigma)
\right).
\end{equation*}
Here, if $z=1$, we understand 
$(1-\frac{\sigma_F}{p^m})\Li_m^\padic(1)=    L_p(m,\omega^{1-m})$. 
\label{thm3.5}
\end{corollary}

\begin{proof}
Theorem \ref{fullformula} (2) and Proposition \ref{prop2.2} 
enable us to derive: 
\begin{align*}
\ell i_m(z)(\sigma)
&=
\frac{(-1)^{m-1}}{(m-1)!}
\tilchi^z_m(\sigma) \\
&=
\frac{-1}{(m-1)!}
\mathrm{Tr}_{F/\Q_p}\bigl(\Li_m^{(p)}(z)\cdot \phi^{CW}_{m,F}(\sigma)
\bigr).
\end{align*}
Recalling $\Li_m^{(p)}(z)=\Li_m^\padic(z)-\frac{1}{p^m}\Li_m^\padic(z^p)$
by definition, where now $z^p=\sigma_F(z)$ for $z\in\muprime$,
we conclude the corollary. 
\end{proof}

%\newpage

\begin{remark}
\label{extensiontoG_F}
Let $P:=G_F^\ab$ and $U$ the image of 
$\mathcal U_\infty(F)$ in $P$.
Then, as in \cite[p.342]{B-K},
$P/U$ is isomorphic to $\hat\Z$ acted on trivially by
$G:=\Gal(F(\zeta_{p^\infty})/F)$ so as to fit in 
the canonical identification
$$
H^1(F,F(m))\cong\Hom_G(P,F(m))\cong\Hom_G(U,F(m))
$$
for $F(m):=F\otimes \Z_p(m)$ $(m\ge 1)$. 
{}From this, we obtain a canonical extension of 
the RHS of the above corollary to a 
$G$-homomorphism from $P$ to $\Q_p(m)$.
On the other hand, if we choose a path 
$\gamma:\overrightarrow{01}\nyoroto z$
so that $\ell i_m(z,\gamma) \in \Hom_G(P,\Q_p(m))$
(this is the case for $\gamma$ in Lemma \ref{specific_path}) 
then, it annihilates an inverse image $u_0$ 
of $1\in\hat\Z$ in $P$. 
Then, these two extended homomorphisms on $P$ 
turn out to coincide with each other,
as they both coincide on $U$ and kill $u_0$.

If moreover $\ell i_m(z,\gamma)$ 
is a 1-cocycle on $G_F$, then they should give the  
same cohomology class in $H^1(F, \Q_p(m))$,
as the restriction map %from $F(\zeta_{p^\infty})$ to $F$ is isomorphic.
from $F$ to $F(\zeta_{p^\infty})$ is injective.

%Notice that, by the ``class formation'' of local class field theory, 
%we may identify the image of $\mathcal U_{\infty}(\Q_p)\to 
%\mathcal U_\infty(F)\overset{\rec_\infty}
%{\longrightarrow}G_{F_\infty}^\ab$ as the inertia subgroup 
%of the image of
%the transfer map 
%$G_{\Q_p(\mu_{p^\infty})}^\ab\to G_{F_\infty}^\ab$.

\end{remark}

%{\bf Question}$\check$ Can we say something for more general 
%$\sigma$ from $\cU_\infty(F)$?

%%%%%%%%%%%%%%%%%%%%%%%%%%%%%%%%%%%%%%%%%%%%%%%%%%%%%%%%%%%%

\section{Appendix: The Kummer level case}
\label{appendix}

%\noindent$\mathbf{A6.1.} \label{appendixP1}$
For completeness, we shall examine the case
$m=0,1$ in Theorem \ref{mainthm}. 
We consider the $\ell$-adic Galois polylogarithms
$\ell i_0$, $\ell i_1$ as the Kummer 1-cocycles
$\kappa_z$, $\kappa_{1-z}$ respectively 
as in (\ref{eq2.4}), (\ref{kummer}) 
 (cf. also \cite{N-W2} \S 5.2).
On the other $p$-adic side, we may regard 
$\Li_0^\padic(z)=-\log_p(z)$, 
$\Li_1^\padic(z)=-\log_p(1-z)$
(cf. \cite{Fu}).
In fact, we have the following: 

\begin{proposition}
\label{m=1}
Let $p$ be an odd prime, and $F$ a finite unramified extension of 
$\Q_p$ with %ring of integers $\cO_F$ and 
the Frobenius substitution $\sigma_F\in\Gal(F/\Q_p)$.
Let $F_\infty:=F(\mu_{p^\infty})$ 
and, for $a\in\cO_F^\times$, let 
$\kappa_{a}:G_{F_\infty}\to\Z_p(1)$ be
the Kummer character for $p$-power roots of
$a$.
Then, 
%for any root of unity $z$ contained in $F$,
%there is a standard specific path $\ovec{01}\nyoroto z$
%such that 
\begin{eqnarray}
\label{maineqn=1}
\kappa_{a}(\sigma)=
%\frac{-1}{(m-1)!}
%-\check
\mathrm{Tr}_{F/\Q_p}
\left(
\left\{
\left(
1-\frac{\sigma_F}{p}
\right)
\log_p(a)
\right\}
\phi^{CW}_{1,F}(\sigma)
\right)
\end{eqnarray}
holds for $\sigma\in G_{F_\infty}$,
%in the image of the reciprocity map
%$\cU_\infty(F){\longrightarrow} G_{F_\infty}^\ab$,
where
$\phi^{CW}_{1,F}:G_{F_\infty}\to F\otimes\Z_p(1)$ 
the 1-st Coates-Wiles homomorphism for the local field $F$
(cf.\,Definition \ref{DefCW}).
\end{proposition}

Following the notation system in \S \ref{sec3.2}, we first
consider the case when $\sigma=\rec_\infty(\epsilon)$ 
for some norm compatible system 
$\epsilon=(\epsilon_n)_n\in \cU_\infty(F)$.
Let $g_\epsilon(T)\in\cO_F\(T\)^{\times}$ 
be the associated Coleman power series such that 
$\epsilon_n=\sigma_F^{-n}g_\epsilon(\zeta_{p^n}-1)$
for $n=1,2,\dots$.

The explicit reciprocity law of S.\,Sen computes the Hilbert
norm residue symbol by the formula${}^{\ast}$
\footnote{($\ast$)
According to \cite[II, p.69, line 12-13]{Sen}, the Hilbert symbol 
in loc. cit. coincides with that in \cite{Iw}. On the other hand,
we employ the Hilbert symbol discussed in Coleman's
papers. According to \cite[p.59, line -6]{C5}, it is the inverse
of the symbol used in \cite{Iw}. Therefore, a minus sign is 
put in the exponent of RHS in ($\ref{SenFormula}$).
}
\begin{equation}
\label{SenFormula}
(\beta,\alpha)_{p^n}=
\zeta_{p^n}^{
-\frac{1}{p^m}
\mathrm{Tr}_{F/\Q_p}
\left(
\log_p(\alpha)
\cdot
\mathrm{Tr}_{F'/F}
(\delta_m(\beta'))
\right)
}
\end{equation}
for $\alpha\in 1+(\zeta_p-1)^2\cO_F$ 
in \cite[Theorem 3 (b)]{Sen} and
more generally for $\alpha\in 1+(\zeta_p-1)\cO_F$
in \cite[II; Theorem 1]{Sen},
where
$F'=F(\zeta_{p^m})$ ($m:$big enough)
and $\beta'\in F'$
are taken so that 
$\alpha^{p^{m-n}}\in 1+(\zeta_p-1)^2\cO_F$ and $N_{F'/F}(\beta')=\beta$.
Finally $\delta_m(\beta')$ means 
$
\frac{\zeta_{p^m}}{g'(\pi)}\cdot\frac{f'(\pi)}{f(\pi)}
$
where $f(T)$, $g(T)\in \cO_F[T]$ are polynomials
with $f(\pi)=\beta$, $g(\pi)=\zeta_{p^m}$ for
any prime $\pi$ of $F'$.
Note that, in our case, we may and do set 
$\pi=\zeta_{p^m}-1$ so that $g'(\pi)=1$
(cf. \cite[Cor.\,15]{C2}).

Now, let us apply (\ref{SenFormula}) for 
$\alpha:=a\in 1+p\cO_F$ and $\beta:=\epsilon_n$.
As $(\alpha,\beta)_{p^n}=(\beta,\alpha)_{p^n}^{-1}$
(cf. \cite[p.352]{C-F}), %, [C2, p.374, line 13-14])
writing $(a,\epsilon_n)_{p^n}=\zeta_{p^n}^{[a,\epsilon_n]}$
and letting $F'=F(\zeta_{p^m})$ as above,
we obtain the following congruence modulo $p^n$:
\begin{equation}
\label{app-arg}
[a,\epsilon_n]\equiv
%-\check
\mathrm{Tr}_{F/\Q_p}
\left(\log_p(a)
\mathrm{Tr}_{F'/F}\left(
\frac{\zeta_{p^m}}{p^m}\cdot
\frac{\sigma_F^{-m}g'_{\epsilon}(\zeta_p^m-1)}{\epsilon_m}
\right)
\right).
\end{equation}

\begin{lemma}
Notations being as above, we have
$$
\mathrm{Tr}_{F'/F}\left(
\frac{\zeta_{p^m}}{p^m}\cdot
\frac{\sigma_F^{-m}g'_{\epsilon}(\zeta_p^m-1)}{\epsilon_m}
\right)
=
\left(1-\frac{\sigma_F^{-1}}{p}\right)\cdot
\phi_{1,F}^{CW}(\epsilon).
$$
\end{lemma}
\begin{proof}
Since $\Norm(g_\epsilon)=\sigma_F g_\epsilon$, we have
$(\sigma_Fg_\epsilon)((1+T)^p-1)=\prod_{i=0}^{p-1}
g_\epsilon(\zeta_p^i(1+T)-1)$.
Substituting $\zeta_{p^k}^j(1+T)-1$ or $(1+T)^{p^k}-1$
( $0\le j\le p-1$, $k=1,2,...$) for $T$ in it,
and combining resulted equations in certain multiple ways,
one obtains
\begin{equation}
(\sigma_F^k g_\epsilon)((1+T)^{p^k}-1)
=\prod_{i=0}^{p^k-1}
g_\epsilon(\zeta_{p^k}^i(1+T)-1) 
\qquad (k\ge 1).
\end{equation}
Taking derivatives of both sides and putting $T=0$, 
we obtain then
\begin{equation}
\label{5.5eq}
p^k\phi_{1,F}^{CW}(\epsilon)
=\sum_{i=0}^{p^k-1}\zeta_{p^k}^i
\frac{\sigma_F^{-k}g'_\epsilon(\zeta_{p^k}^i-1)}
{\sigma_F^{-k}g_\epsilon(\zeta_{p^k}^i-1)}
=
\sum_{i=0}^{p^k-1}\zeta_{p^k}^i
\frac{\sigma_F^{-k}g'_\epsilon(\zeta_{p^k}^i-1)}
{\epsilon_k}.
\end{equation}
Using (\ref{5.5eq}) for $k=m$, $m-1$, we see that
the LHS of Lemma equals to
\begin{align*}
&\frac{1}{p^m}\sum_{\substack{0\le i<p^m\\ p\nmid i}}
\frac{\sigma_F^{-m}g'_\epsilon(\zeta_{p^m}^i-1)}
{\sigma_F^{-m}g_\epsilon(\zeta_{p^m}^i-1)}  \\
=&
\frac{1}{p^m}
\left(\sum_{i=0}^{p^m}
\frac{\sigma_F^{-m}g'_\epsilon(\zeta_{p^m}^i-1)}
{\sigma_F^{-m}g_\epsilon(\zeta_{p^m}^i-1)}
-
\sum_{j=0}^{p^{m-1}}
\sigma_F^{-1}\Bigl(
\frac{\sigma_F^{-(m-1)}g'_\epsilon(\zeta_{p^{m}}^{pj}-1)}
{\sigma_F^{-(m-1)}g_\epsilon(\zeta_{p^m}^{pj}-1)}
\Bigr)\right) 
 \\
=&
\phi_{1,F}^{CW}(\epsilon)
-
\frac{1}{\,p\,}\,
\sigma_F^{-1}\Bigl(
\phi_{1,F}^{CW}(\epsilon)\Bigr).
\end{align*}
This completes the proof of the lemma.
\end{proof}

Plugging this lemma into (\ref{app-arg}), we find
\begin{align*}
[a,\epsilon_n]
&\equiv
\mathrm{Tr}_{F/\Q_p}
\left(\log_p(a)\phi_{1,F}^{CW}(\epsilon)
-\frac{\log_p(a)}{p}
\sigma_F^{-1}\Bigl(\phi_{1,F}^{CW}(\epsilon)\Bigr)
\right) \\
&\equiv\mathrm{Tr}_{F/\Q_p}
\left(\log_p(a)\phi_{1,F}^{CW}(\epsilon)
-\sigma_F^{-1}\left\{
\sigma_F\Bigl(
\frac{\log_p(a)}{p}
\phi_{1,F}^{CW}(\epsilon)
\Bigr)\right\}
\right) \\
&\equiv
\mathrm{Tr}_{F/\Q_p}
\left(\log_p(a)\Bigl(1-\frac{\sigma_F}{p}\Bigr)
\phi_{1,F}^{CW}(\epsilon)
\right) 
\qquad (\mathrm{mod}\ p^n).
\end{align*}
This settles Proposition \ref{m=1} in the case 
$a\in 1+p\cO_F$ and $\sigma=\rec_\infty(\epsilon)$,
since by definition $\kappa_a(\sigma)\equiv [a,\epsilon_n]$
mod $p^n$.
The general case for $a\in \cO^\times$ follows 
immediately from observing that 
$\log_p(a)=\log_p(a')$ when $a=\zeta a'$ 
for any $\zeta\in\mu_p$.
The extension of the statement to all 
$\sigma\in G_{F(\zeta_{p^\infty})}$
follows from Remark \ref{extensiontoG_F}.
This settles the proof of Proposition \ref{m=1}. 
\qed

%%%%%%%%%%%%%%%%%%%%%%%%%%%%%%%%%%%%%%%%%%%%%%%%%%%%%%%%%%%%%%%%%%%%%%%%%5
%\newpage
%%%%%%%%%%%%%%%%%%%%%%%%%%%%%%%%%%%%%%%%%%%%%%


\begin{thebibliography}{99}

\bibitem[A]{A} G. Anderson,
{\em The hyperadelic gamma function}, 
{Invent. Math.}, {\bf 95} (1989), 63--131.

\bibitem[B-K]{B-K} S. Bloch and K. Kato, 
{\em $L$-functions and Tamagawa numbers of motives}, The Grothendieck Festschrift  Vol. I,
(P.Cartier et. al. eds.), {Progr. Math.}, {\bf 86} (1990), 
333--400.

\bibitem[C1]{C1}R. Coleman, {\em Division values in local fields}, 
{Invent. Math.} {\bf 53} (1979), 91--116.

\bibitem[C2]{C2}R. Coleman, {\em The dilogarithm and the norm residue symbol}, 
{Bull. Soc. Math. France} {\bf 109} (1981), 
373--402. 

\bibitem[C3]{C3}R. Coleman, {\em Dilogarithms, regulators, and $p$-adic L-functions}, 
{Invent. Math.}, {\bf 69} (1982), 171--208.

\bibitem[C4]{C4}R. Coleman, {\em Local units modulo circular units}, 
{Proc. Amer. Math. Soc.} {\bf 89} (1983), 1--7. 

\bibitem[C5]{C5}R. Coleman, {\em Anderson-Ihara theory: Gauss sums and circular units}, 
in {``Algebraic number theory -- in honor of K.Iwasawa''}
(J.Coates, R.Greenberg, B.Mazur, I.Satake eds.), 
{Adv. Studies in Pure Math.}, {\bf 17} (1989), 55--72. 

\bibitem[C-F]{C-F}
J. W. S. Cassels, A. Fr\"ohlich,
{\em Algebraic Number Theory}, 2nd edition (2010),
London Mathematical Society.

\bibitem[C-S]{C-S}J. Coates and R. Sujatha, Cyclotomic fields and zeta values, Springer-Verlag, Berlin, 2006. 

\bibitem[De]{De} P. Deligne, 
{\em Le groupe fondamental de la droite projective moins trois points}, 
in {``Galois groups over $\Q$''} (Y.Ihara, K.Ribet, J.-P.Serre eds.), 
{Math. Sci. Res. Inst. Publ.}, {\bf 16} (1989), 79--297.


\bibitem[I-S]{I-S}F. Ichimura and K. Sakaguchi, 
{\em The nonvanishing of a certain Kummer character $\chi_m$ (after C. Soul\'e), 
and some related topics}, 
in {``Galois representations and arithmetic algebraic geometry''}
(Y. Ihara ed.), 
 {Adv. Studies in Pure Math.}, {\bf 12} (1987), 53--64.


\bibitem[I86]{I86} Y. Ihara, 
{\em Profinite braid groups, Galois representations and complex multiplications}, 
{Ann. of Math.} (2) {\bf 123} (1986), 43--106.

\bibitem[I90]{I90} Y. Ihara, 
{\em Braids, Galois groups, and some arithmetic functions},
Proc. Intern. Congress of Math. Kyoto 1990, 99--120.


\bibitem[IKY]{IKY} Y. Ihara, M. Kaneko, A. Yukinari, 
{\em On some properties of the universal power series for Jacobi sums}, 
in {``Galois representations and arithmetic algebraic geometry''}
(Y. Ihara ed.), 
 {Adv. Studies in Pure Math.}, {\bf 12} (1987),
65--86.

\bibitem[Iw]{Iw} K. Iwasawa,
{\em On explicit formulas for the norm residue symbol},
J. Math. Soc. Japan, {\bf 20} (1968), 151--165.

\bibitem[Fu]{Fu}H. Furusho, {\em $p$-adic multiple zeta values 
I: $p$-adic multiple polylogarithms and the $p$-adic KZ equation}, 
{Invent. Math.}, {\bf 155} (2004), 223--286;
{\em II: Tannakian interpretations}, {Amer. Journal of Math.}, {\bf 129} (2007), 1105--1144.

\bibitem[Gr]{Gr} M. Gros, 
{\em R\'egulateurs syntomiques et valeurs de fonctions $L$ $p$-adiques, I}, 
(with an appendix by Masato Kurihara), 
{Invent. Math.}, {\bf  99} (1990), 293--320; 
{\em II}, {Invent. Math.}, {\bf 115} (1994), 61--79. 


\bibitem[Kim]{Kim} M. Kim, {\em The unipotent Albanese map and Selmer varieties for curves}, 
{Publ. Res. Inst. Math. Sci.}, {\bf 45} (2009), 89--133.

\bibitem[Ko]{Ko}N. Koblitz,
{\em  A new proof of certain formulas for $p$-adic L-functions}, 
{Duke Math. J.} {\bf 46} (1979),
455--468.

\bibitem[KN]{KN} M. Kolster, T. Nguyen Quang Do,
{\em Syntomic regulators and special values of $p$-adic $L$-functions},
{Invent. math.}, {\bf 133} (1998), 417--447.


\bibitem[Ku]{Ku} M. Kurihara, 
{\em Computation of the syntomic regulator in the cyclotomic case}, Appendix to [Gr, I].


\bibitem[NSW]{NSW} H. Nakamura, K. Sakugawa, Z.Wojtkowiak, 
{\em Polylogarithmic analogue of the Coleman-Ihara formula, II},
{\it Preprint avaliable at} : 
https:/\!/sites.google.com/site/kenjisakugawashomepage/reserch-articles


\bibitem[N-W1]{N-W1} H. Nakamura and Z. Wojtkowiak, 
{\em On explicit formulae for l-adic polylogarithms}, 
in {``Arithmetic fundamental groups and noncommutative algebra''},
(M.Fried, Y.Ihara eds.)
{Proc. Sympos. Pure Math.}, {\bf 70} (2002), 285--294.

\bibitem[N-W2]{N-W2} H. Nakamura and Z. Wojtkowiak, 
{\em Tensor and homotopy criteria for functional equations of $l$-adic and classical iterated integrals}, 
in {``Non-abelian fundamental groups and Iwasawa theory''}
(J.Coates, M.Kim, F.Pop, M.Sa\"idi, P.Schneider eds.), 
{London Math. Soc. Lecture Note Ser.}, {\bf 393} (2012), 258--310.

\bibitem[Ol]{Ol} M. C. Olsson,
{\em Towards Non-Abelian p-adic Hodge Theory in the Good Reduction Case},
Memoirs of A.M.S. {\bf 210}, 2011.

\bibitem[Sak]{Sak} K. Sakugawa, 
{\em On non-abelian generalization of the Bloch-Kato
exponential map}, 
{\it Preprint avaliable at} : 
https:/\!/sites.google.com/site/kenjisakugawashomepage/reserch-articles

\bibitem[Sen]{Sen} S. Sen,
{\em On explicit reciprocity laws},
J. reine anew. math. {\bf 313} (1980), 1--26;
{\em Part II},  
J. reine anew. math. {\bf 323} (1981), 68--87.

\bibitem[Sou]{Sou} C. Soul\'e,
{\em On higher p-adic regulators}, 
Lecture Notes in Mathematics, {\bf  854} (1981), 
372--401.

\bibitem[Wa]{Wa} L. C. Washington, Introduction to Cyclotomic Fields, 2nd Edition,
Springer 1997.

\bibitem[W1]{W1}  Z. Wojtkowiak, 
{\em On $l$-adic iterated integrals I: Analog of Zagier conjecture}, {Nagoya Math. J.} 
{\bf 176} (2004), 113--158.

\bibitem[W2]{W2}  Z. Wojtkowiak, 
{\em On $l$-adic iterated integrals II: Functional equations and l-adic polylogarithms}, 
{Nagoya Math. J.} {\bf 177} (2005), 117--153. 

\bibitem[W3]{W3}Z. Wojtkowiak, {\em On $l$-adic Galois L-functions}, (preprint 2014) 
arXiv:1403.2209.


\end{thebibliography}
\end{document}